\documentclass[11pt]{amsart}
\usepackage{amssymb,amsmath,amsfonts,amscd,euscript}
\usepackage[matrix,arrow,curve]{xy}
\usepackage[backref=page]{hyperref}
\usepackage{tikz-cd}
\usepackage{fullpage}

\newcommand{\nc}{\newcommand}
\newtheorem{theoremA}{Theorem}

\numberwithin{equation}{section}

\newtheorem{thm}[equation]{Theorem}
\newtheorem{prop}[equation]{Proposition}
\newtheorem{lem}[equation]{Lemma}
\newtheorem{cor}[equation]{Corollary}
\newtheorem{rem}[equation]{Remark}

\newtheorem{example}[equation]{Example}
\newtheorem{dfn}[equation]{Definition}

\nc{\fB}{\mathfrak{B}}
\nc{\gl}{\mathfrak{gl}}
\nc{\GL}{\mathfrak{GL}}
\nc{\g}{\mathfrak{g}}
\nc{\gh}{\widehat\g}
\nc{\h}{\mathfrak{h}}
\nc{\wfh}{\widehat{\mathfrak{h}}}
\nc{\la}{\lambda}
\nc{\al}{\alpha }
\nc{\be}{\beta }
\nc{\ve}{\varepsilon }
\nc{\om}{\omega }
\nc{\lr}{\text{-}}
\nc{\ta}{\theta}
\nc{\ch}{{\mathop {\rm ch}}}
\nc{\Tr}{{\mathop {\rm Tr}\,}}
\nc{\Id}{{\mathop {\rm Id}}}
\nc{\ad}{{\mathop {\rm ad}}}
\nc{\coker}{{\mathop {\rm coker}}}
\nc{\bra}{\langle}
\nc{\ket}{\rangle}
\nc{\bi}{{\bf i}}
\nc{\pa}{\partial}
\nc{\ld}{\ldots}
\nc{\cd}{\cdots}
\nc{\hk}{\hookrightarrow}
\nc{\T}{\otimes}
\nc{\gr}{\mathrm{gr}}
\nc{\ov}{\overline}

\nc{\msl}{\mathfrak{sl}}
\nc{\mgl}{\mathfrak{gl}}
\nc{\U}{\mathrm U}
\nc{\V}{\EuScript V}
\nc{\cO}{\mathcal{O}}
\nc{\cL}{\mathcal{L}}
\nc{\Res}{\mathrm{Res\ }}

\newcommand{\bC}{{\Bbbk}}

\newcommand{\cD}{{\mathcal D}}

\newcommand{\bZ}{{\mathbb Z}}

\newcommand{\bI}{{\mathbb I}}
\newcommand{\bP}{{\mathbb P}}

\newcommand{\fh}{{\mathfrak h}}

\newcommand{\fs}{{\mathfrak s}}
\newcommand{\fg}{{\mathfrak g}}

\newcommand{\fb}{{\mathfrak b}}

\newcommand{\cF}{{\mathcal{F}}}
\newcommand{\fn}{{\mathfrak n}}

\newcommand{\Hom}{\mathrm{Hom}}
\newcommand{\RHom}{\mathrm{RHom}}
\newcommand{\wt}{\mathrm{wt}}
\newcommand{\qwt}{\mathrm{qwt}}
\renewcommand{\dir}{\mathrm{dir}}

\newcommand{\id}{\mathrm{id}}
\nc{\I}{\mathfrak I}
\nc{\bfI}{\mathbf I}
\newcommand{\fC}{{\mathfrak C}}
\nc{\Q}{\mathfrak Q}
\nc{\fr}{\mathfrak r}
\nc{\W}{\mathbb W}
\nc{\bU}{\mathbb U}

\nc{\Gm}{\mathbb{G}_{m}}
\nc{\bA}{\mathbb A}

\newcommand{\ZZ}{\mathbb{Z}}
\newcommand{\QQ}{\mathbb{Q}}

\newcommand{\cA}{\mathcal{A}}

\newcommand{\de}{\text{-}}

\DeclareMathOperator{\Ext}{Ext}

\newcommand{\ldot}{{\:\raisebox{1.5pt}{\selectfont\text{\circle*{1.5}}}}}
\newcommand{\udot}{{\:\raisebox{4pt}{\selectfont\text{\circle*{1.5}}}}}

\let\leq\leqslant
\let\geq\geqslant
\newcommand{\ttt}{\text{-}}

\usepackage{ulem}
\usepackage{color}

\title
{Peter-Weyl theorem for Iwahori groups and highest weight categories}

\author{Evgeny Feigin}
\address{Evgeny Feigin:\newline
School of Mathematical Sciences, Tel Aviv University, Tel Aviv
69978, Israel
}
\email{evgfeig@gmail.com}

\author[Khoroshkin]{Anton Khoroshkin$\phantom{}^{*}$}
\thanks{$\phantom{}^{*}$Corresponding author}
\address{Anton Khoroshkin: \newline
Department of Mathematics, University of Haifa, Mount Carmel, 3103301, Haifa, Israel
}

\email{khoroshkin@gmail.com}

\author[Makedonskyi]{Ievgen Makedonskyi}
\address{Ievgen Makedonskyi:\newline
Yanqi Lake Beijing Institute of Mathematical Sciences And Applications (BIMSA), 
No. 544, Hefangkou Village, Huaibei Town, Huairou District, Beijing 101408.}
\email{makedonskii\_e@mail.ru}

\author[Orr]{Daniel Orr}
\address{Daniel Orr:\newline  
Department of Mathematics, Virginia Tech, 225 Stanger St., Blacksburg,
VA 24061 USA\newline
Max Planck Institute for Mathematics, Vivatsgasse 7, 53111 Bonn, Germany}
\email{dorr@math.vt.edu}

\begin{document}

\begin{abstract}
We study the algebra of functions on the Iwahori group via the category of graded bounded representations of the Iwahori Lie algebra. In particular, we identify the standard and costandard objects in this category with certain generalized Weyl modules. 
Using this identification we express the characters of the standard and costandard objects in terms of specialized nonsymmetric Macdonald polynomials. 
We also prove that our category of interest admits a generalized highest weight structure (known as stratified structure). We show, more generally,
that such a structure on a category of representations of a Lie algebra implies the Peter-Weyl type theorem for the corresponding algebraic group. In the Iwahori case, standard filtrations of indecomposable projective objects correspond to new ``reciprocal'' Macdonald-type identities. 
\end{abstract}

\maketitle

\setcounter{tocdepth}{1}
\tableofcontents

\setcounter{section}{-1}
\section{Introduction}
The main goal of this paper is to study the algebra of functions on the Iwahori subgroup of an affine Kac-Moody group. The cases of the affine Lie algebras associated with $\mgl_n$ and $\msl_n$ were considered in \cite{FMO2}.
From one hand, the type $A$ case is the most natural starting point due to numerous applications in combinatorics and in classical representation theory
(such as Cauchy identities, Schur polynomials and spaces of functions on matrix spaces). On the other hand, the $\mgl_n$ case, which is the most ``classical'', does not 
completely fit into the realm of affine Kac-Moody Lie algebras. Moreover, the
methods and tools developed in \cite{FMO2} are somewhat special and do not allow one to study the general case. The main new feature of the present paper is the use of categorical methods, extending those of \cite{BBCKhL,CI} to the Iwahori setting. More precisely, we study certain categories of representations of the Iwahori algebra
and describe standard and costandard objects. This turns out to be the key for the description of the space of functions on the Iwahori group. 

Let us describe our results in more detail.
Let $S$ be a simple finite-dimensional algebraic group and $\fs$ be the corresponding Lie algebra. Then the Peter-Weyl theorem  \cite{PW,TY}  describes the bimodule of functions 
$\Bbbk[S]$ as the direct sum $\bigoplus_{\la\in P_+} L_\la\T L_\la^o$, where $P_+$ is the set of dominant integral weights, $L_\la$ is the (left) irreducible $\fs$-module  and $L_\la^o$ is the irreducible right module. A similar (nonsymmetric) problem for $\Bbbk[B]$, where $B\subset S$ is a Borel subgroup, was considered by van der Kallen in \cite{vdK} (see also \cite{BN} for the quantum affine case). 
He proved that there exists a filtration on $\Bbbk[B]$ such that the homogeneous pieces of the associated graded space decompose
as tensor products of Demazure and van der Kallen modules (quotients of the Demazure modules by the sum of all smaller Demazure submodules). The main motivation for this work was to extend van der Kallen's result to the Iwahori subgroup  $\bfI$, which is, of course, the affine analog of the Borel subgroup. More precisely, $\bfI\subset S[[t]]$ is the preimage of $B$ with respect to the $t=0$ specialization map $S[[t]]\to S$.
The bimodule of functions $\Bbbk[S[[t]]]$ on the current group was described in \cite{FKhM}. 

Let $\mathcal{I}$ be the Iwahori Lie algebra.
In order to study the space $\Bbbk[\bfI]$ as an $\mathcal{I}\de \mathcal{I}$ bimodule we need to understand the category of graded representations of $\mathcal{I}$; we note that 
categorical arguments also play crucial role in the van der Kallen approach \cite{vdK} (see also \cite{ChKa}). We start with a more general situation of (possibly infinite-dimensional) graded Lie algebra $\fg$ satisfying certain technical assumptions (see section \ref{progroup}). In particular, $\fg$ is a direct sum of a reductive subalgebra $\fg_0$ and a graded ideal $\fr$. 
In particular, for $\fg=\mathcal{I}$  the reductive subalgebra is $\fg_0=\fh$, the Cartan subalgebra of $\fs$.  We consider the category $\fC$ of graded bounded $\fg$-modules and apply a general categorical result to the description of the space of functions.

More precisely, we introduce the notion of a stratified category, which is an upgrade of the notion of the highest weight category in the graded situation (see \cite{CG,CPS,Kl,LW,Kh} for the foundations and applications of the highest weight categories and \cite{B,BWW} for the graded version in different settings). 
Recall that in \cite{Kh} basic properties of the highest weight categories were derived; in particular, the BGG reciprocity relating certain
multiplicities in projective (injective) and costandard (standard) modules was established (see \cite{BBCKhL,CI,CFK} for the representations of current algebras). In \cite{Kh}, a criterion for the category to be stratified was given. We give the generalization of this criterion (see Theorem \ref{thm::Stratified}).

In order to formulate our first theorem, we need the following notation. Let   
$P_+=P_+(\fg_0)$ be the set of dominant integral weights for the reductive subalgebra $\fg_0\subset\fg$. Then the simple modules in $\fC$ are indexed by $P_+\times\ZZ$. For $(\la,k)\in P_+\times\ZZ$, let $\nabla_{(\la,k)}$ and $\Delta_{(\la,k)}$ be costandard and standard objects, respectively. Often we only need to work with $\nabla_\la := \nabla_{(\la,0)}$ and $\Delta_\la=\Delta_{(\la,0)}$. We denote by $\cA_\la$ the 
algebra of endomorphisms of $\nabla_\la$, which we assume to coincide with the algebra of endomorphisms of the right module $(\Delta_\la^\vee)^o$, where $\vee$ denotes the restricted dual and $o$ denotes twisting by the anti-automorphism $x\mapsto-x$ of $\fg$.

\begin{theoremA}(Theorem \ref{FT})
\label{thmA}
If the category $\fC$ of graded representations of a Lie algebra $\fg$ (satisfying technical assumptions from section \ref{progroup}) is stratified, then the algebra of functions $\Bbbk[G]$ on the corresponding simply-connected algebraic group $G$ admits a filtration of $G$--$G$ bimodules such that 
\[
{\rm gr}\ \! \Bbbk[G] \simeq \bigoplus_{\la\in P_+} \nabla_\la\T_{\cA_\la} (\Delta_\la^\vee)^o.
\]
\end{theoremA}

Now let us focus on the case of the Iwahori algebra. In this case, we are able to identify standard and costandard objects
with certain cyclic $\mathcal{I}$-modules  \cite{FeMa,FKM,Ka} by an explicit presentation; these are the generalized Weyl modules with characteristics. (We note that Weyl modules have proved to be very useful in various problems of geometric representation theory; see \cite{BF,FKhM}.) 
This identification allows us to compute the characters of the modules $\Delta_\la$ and $\nabla_\la$ in terms of the nonsymmetric Macdonald polynomials (\cite{Ch1,Ch2,OS,FMO1,I}). Here we note that the set $P_+$ coincides with the full weight lattice of $\fs$, i.e., $P_+(\fg_0)=P(\fs)$.

We use also the proper standard and proper costandard modules
$\overline\Delta_\la ,\overline\nabla_\la$ which are ``local'' versions of $\Delta_\la ,\nabla_\la$. The characters of proper modules differ from the characters of $\Delta_\la$ and $\nabla_\la$ by the character of their endomorphism algebras. 
We note that a crucial step for checking that a category is stratified
is the computation of Ext spaces between (proper) standard and costandard objects. In the simply laced case this computation can be
found in \cite{FKM}, see also \cite{ChKa}. Our approach works in the general situation.

\begin{theoremA}
\label{thmB}
Let $\fC$ be the category of graded representations of the Iwahori algebra. Then:
\begin{enumerate}
\item (Theorem \ref{DUStCost})
$\overline\Delta_\la$ and $\overline\nabla_\la^\vee$ are isomorphic to generalized Weyl modules (with characteristics).
\item   (Theorem \ref{Dcharacter}, Proposition \ref{UUo}) One has the character identities
\[
\ch\,\overline\Delta_\la = E_\la(x,q,0),\qquad
\ch\,\overline\nabla_\la^o = E_\la(x,q^{-1},\infty),
\]
where $E_\la(x,q,t)$ are the non-symmetric Macdonald polynomials.
\item (Theorem \ref{TheoremStr}) $\fC$ is stratified.
\end{enumerate}   
\end{theoremA}

Note that the character formula for the generalized Weyl module which is isomorphic to $\overline\nabla_\la$ was proven in \cite{FKM}, and the character formula for $\overline\Delta_\la$ in types $ADE$ was proven in \cite{Sa,I} (in these types they are isomorphic to affine Demazure modules). 

The existence of a filtration on $\Bbbk[\bfI]$ such that the
homogeneous pieces are isomorphic to the tensor products over the highest weight algebras of certain 
generalized Weyl modules was conjectured in \cite[Conjecture 1]{FMO2}. Our Theorems~\ref{thmA} and \ref{thmB} settle this conjecture. We also prove another conjecture from \cite{FMO2}. 
Namely, let $\bP_\la$ and $L_\la$ for $\la\in P(\fs)$ be the indecomposable projective and simple representations of the Iwahori algebra (with minimal degree $0$). The following theorem can be understood as the equality of two different expressions for the quadratic ``Casimir operator'' and settles \cite[Conjecture 2]{FMO2}; we note that the equality requires genuinely infinite linear combinations, unlike the Cauchy kernel for symmetric functions (cf. \cite[Ch. I, (4.6)]{M1}).
\begin{theoremA} (Theorem \ref{conj})
\label{thmC}
One has $\sum_{\la\in P(\fs)} \ch(\bP_\la \otimes L_\la) = \sum_{\la\in P(\fs)} \ch(\Delta_\la\otimes \overline\nabla_\la^o)$, where both sides are well-defined series in $q$ and (multivariables) $x$ and $y$ and are equal to the graded bicharacter of $\Bbbk[\bfI]^\vee$. More concretely,
\begin{equation}
\label{IntroBiCharacter}
\frac{\sum_{\la\in P(\fs)} x^\la y^\la}{(q;q)_\infty^{\mathrm{rk}(\fs)}\prod_{\al\in\Delta_+} (1-x^\al)\prod_{\al\in\Delta} (qx^\al;q)_\infty}=\\
\sum_{\lambda\in P(\fs)} 
(q)_\la^{-1} E_{\lambda}(x,q,0)E_{\lambda}(y, q^{-1},\infty)
\end{equation}
where $(z;q)_\infty=(1-z)(1-qz)(1-q^2z)\dotsm$ and $(q)_\la$ is defined by \eqref{Norm}.
\end{theoremA}

Taking the coefficient of $y^\la$ on both sides of \eqref{IntroBiCharacter}, we arrive at ``reciprocal'' Macdonald-type identities (see \eqref{PIntoE} and \eqref{ProjIntoSt2}) expressing each indecomposable projective character $\ch\,\bP_\la$ as an infinite linear combination $\sum_{\mu} m_{\la\mu}(q)\ch\,\Delta_\mu$ of characters of standard objects with non-negative coefficients $m_{\la\mu}(q)\in\ZZ_{\ge 0}[q]$. Our work gives the following categorical explanation for this non-negativity: the coefficient $m_{\la\mu}(q)$ is exactly the graded multiplicity $\Delta_\mu$ in an excellent $\Delta$-filtration of $\bP_\la$. That such filtrations exist is a consequence of the stratified structure. 
We note that, by BGG reciprocity (see Corollary~\ref{CorBGG}), $m_{\la\mu}(q)$ is also the graded multiplicity of the simple $L_\la$ in the proper costandard object $\overline\nabla_\mu$.

In type $A$, the identity \eqref{IntroBiCharacter} is essentially the $t=0$ specialization of the Cauchy-type identity due to Mimachi-Noumi \cite{MN}, which was the starting point of our approach in \cite{FMO2}. Our Theorem C gives a new proof of the $t=0$ specialized Mimachi-Noumi identity and extends it to arbitrary types. Specializing \eqref{IntroBiCharacter} to $q=0$, one arrives at a nonsymmetric Cauchy-type identity which is the combinatorial shadow of van der Kallen's result \cite{vdK} mentioned above; here we note that the specialized polynomials $E_\la(x,0,0)$ and $E_\la(x,\infty,\infty)$ coincide with the (finite-dimensional) Demazure characters and Demazure atoms, respectively. In type $A$, the $q=t=0$ specialization of \eqref{IntroBiCharacter} (or the Mimachi-Noumi formula) gives a special case of Lascoux's nonsymmetric Cauchy identities \cite{L}, which have been studied in several works including \cite{ChKw,AE,AGL}. Our work incorporates the parameter $q$ and gives categorical meaning to the identity \eqref{IntroBiCharacter} in arbitrary type. We anticipate a crystal-theoretic interpretation of \eqref{IntroBiCharacter} along the lines of \cite{IMS,KL} but do not take this approach in the present work.

This paper is organized as follows. In Section \ref{HWC} we develop the theory of the stratified categories.
In Section \ref{BimoduleFunctions} we study the bimodule of functions on an algebraic group satisfying certain assumptions that are outlined in Subsection~\ref{progroup} and prove Theorem~\ref{thmA}.
In Section \ref{MacdonaldPolynomials} we collect basics on the nonsymmetric Macdonald polynomials and related combinatorics. 
In Section \ref{Iwahori} we study certain modules and bimodules over the Iwahori algebra, and in Section \ref{IwahoriStratified} we prove that categories of representations defined in \S\ref{progroup} for the Iwahori algebra are stratified, leading to proofs of Theorems~\ref{thmB} and \ref{thmC}.

\subsection*{Acknowledgments}
We thank Syu Kato and Satoshi Naito for their helpful discussions. 
The work of Ie. M. partially supported by the International Scientists Program of the Beijing Natural Sciences
Foundation (grant number IS24002).

\section{Highest weight categories and stratified categories}\label{HWC}

The goal of this section is to review the homological framework that will be used in the main body of the text for the category of representations of the Iwahori Lie algebra.  
Most of this material is standard, but we include it here to make the text more self-contained. The discussion below begins with an arbitrary abelian category \( \mathcal{O} \). However, since the main focus of this paper is on categories of representations of various algebras, we refer to the objects of \( \mathcal{O} \) as modules.  

For our applications, we need to consider two dual situations. In the first case, we work with an abelian category that has enough projective objects, while in the second case, the category has enough injective objects.  
We first formulate all the relevant results, providing full proofs in the case of projective modules, and then outline the corresponding statements for the dual "injective" case.

\subsection{Ordering irreducibles and Serre subcategories}

Let $\Theta$ be a lower finite partially ordered set (i.e. for any $\mu\in\Theta$ there are finitely many elements smaller than $\mu$). Let $\cO$ be an abelian $\Bbbk$-linear category
with simple objects $L_{\lambda}$ labeled by elements $\lambda$ of a set $\Lambda$ 
(in all applications we work over algebraically closed field  $\Bbbk$ of characteristic zero). 

\begin{dfn}\label{Ordered:Category}
We call an abelian category  $\cO$ $\Theta$-ordered if 
\begin{itemize}
\item $\cO$ enjoys the Krull-Schmidt property;
\item the multiplicity $[M:L_{\lambda}]$ of each simple object $L_{\lambda}$ in any object $M$ of $\cO$ is finite;
\item $\cO$ is equipped with the map $\rho:\Lambda\to \Theta$.
\end{itemize}
\end{dfn}
For each $\lambda\in\Lambda$ we define the Serre subcategory $\cO_{\leq\lambda}$ in $\cO$ to be a full subcategory generated by  irreducibles $L_\mu$ with $\rho(\mu)\leq\rho(\lambda)$:
$$
M\in \cO_{\leq\lambda} \Leftrightarrow [M:L_{\mu}]=0 \text{ whenever } \rho(\mu)\not\leq\rho(\lambda).
$$
The fully faithful embedding $\imath_{\leq\lambda}:\cO_{\leq\lambda} \rightarrow \cO$ is exact and admits a left adjoint functor $\imath_{\leq\lambda}^*$ and the right adjoint functor $\imath_{\leq\lambda}^!$:
\begin{equation}\label{diagram}
\begin{tikzcd}
\cO_{\leq\lambda}
\arrow[rrr,"\imath_{\leq\lambda}"description,"\perp" near start,"\perp"' near start]
&&&
\cO
\arrow[lll,shift left=2ex,"\imath_{\leq\lambda}^!(M):=M_{\leq\lambda}"]
\arrow[lll,shift right=2ex,"\imath_{\leq\lambda}^*(M):=M/((M)_{\not\leq\lambda})"']
\end{tikzcd}
\end{equation}
where $\imath_{\leq\lambda}^*(M)$ is the maximal quotient module of $M$ that belongs to the subcategory $\cO_{\leq\lambda}$ and $\imath_{\leq\lambda}^{!}M$ is the maximal submodule of $M$ that belongs to $\cO_{\leq\lambda}$.

Similarly, for any finite subset $S\subset\Lambda$ there exists a Serre subcategory $\cO_{\leq S}$, an exact fully faithful embedding $\imath_{\leq S}:\cO_{\leq S}\to \cO$, and its adjoints $\imath_{\leq S}^{*},\imath_{\leq S}^{!}$. 
For each simple object $L_\mu$ showing up in the Krull-Schmidt decomposition of $M\in\cO_{\leq S}$ there exists $s\in S$ such that $\rho(\mu)\leq\rho(s)$.

\begin{rem}
\label{remark::filtration}
Note that for $M \in \cO$ we have two filtrations indexed by $\Theta$: the increasing filtration $\imath_{\leq\lambda}^!(M):=M_{\leq\lambda}$ and the decreasing filtration $((M)_{\not\leq\lambda})$. In other words,  for $\mu \leq \lambda$ we have:
\begin{gather}
\label{eq::filtr::sub}
  \imath_{\leq\mu}^!(M)\subset \imath_{\leq\lambda}^!(M);
\\
\label{eq::filtr::factor}
((M)_{\not\leq\mu})\supset((M)_{\not\leq\lambda}) \ \Leftrightarrow \ \imath_{\leq\lambda}^{*}(M) \twoheadrightarrow \imath_{\leq\mu}^{*}(M).   
\end{gather}
Without loss of generality, we assume that these filtrations are exhaustive and separated:
$$
\bigcap_{\lambda\in\Lambda} \imath_{\leq\lambda}^!(M) = \bigcap_{\lambda\in\Lambda} (M)_{\not\leq\lambda} = 0; \quad 
\bigcup_{\lambda\in\Lambda} \imath_{\leq\lambda}^!(M) = \bigcup_{\lambda\in\Lambda} (M)_{\not\leq\lambda} = M.
$$
\end{rem}

\begin{rem}
Assume that the category $\cO$ consists of representations of a Lie algebra with a fixed Cartan subalgebra. 
Moreover, assume that $\Lambda$ is the set of weights for the Cartan subalgebra
and each $L_\lambda$ is one-dimensional. Then each $M$ (as a module over the Cartan subalgebra) is a direct sum of simples and
$\imath_{\leq\lambda}^*(M)$ is the quotient of $M$ by the submodule generated by the $L_{\mu}$ with $\rho(\mu)\not\leq \rho(\lambda)$.
\end{rem}

In this paper we focus on categories with an infinite number of simples and modules we are considering usually have an infinite collection of simple subquotients, so let us be a bit more precise.
Namely, we are working with one of the two possibilities for the category $\cO$: 
\begin{itemize}
    \item $\cO$ has enough injectives and each module $M\in\cO$ admits an increasing exhaustive filtration $0=F_0M \subset F_1 M \subset \ldots$, such that the consecutive subquotients $F_iM/F_{i-1}M$ are simple and each simple appears as a subquotient finite number of times.
    
    \item $\cO$ has enough projectives and each module $M\in\cO$ has a decreasing separated filtration $M=F_0 M\supset F_1 M\supset \ldots$, such that the consecutive subquotients $F_iM/F_{i+1}M$ are simple and each simple appears as a subquotient finite number of times.
\end{itemize}

Assuming the existence of the filtrations as above, the Krull-Schmidt property means that for any given object $M$ the multiplicity of a given simple as a component of the associated graded does not depend on the filtration. In other words, one can assign to each object $M$ in a $\Theta$-ordered category $\cO$ its character which is a formal sum of simples with their multiplicities:
\begin{equation}
\label{eq::K0::map}    
M \mapsto \sum_{i=0}^{\infty} [F_iM/F_{i+1}M] = \sum_{\lambda\in \Lambda} m_{\lambda}[L_{\lambda}], \quad m_{\lambda}:=\#\{i\colon F_iM/F_{i+1}M \simeq L_{\lambda}\}.
\end{equation}
Note that this sum might be infinite, but we still denote the corresponding abelian group of characters by $K_0(\cO)$. We call the partially ordered set $\Theta$ the set of weights. 
We denote the localization map~\eqref{eq::K0::map} by $M\mapsto [M]$.
 Note that thanks to the Krull-Schmidt property we know that 
$$M=0 \text{ in } \cO \Leftrightarrow [M]=0 \text{ in } K_0(\cO).$$

Let us also introduce one of the main "protagonists" of this paper.
\begin{dfn}\label{StandardCostandard}
Let $\la\in\Lambda$.
\begin{itemize}
    \item 
    The standard module $\Delta_{\lambda}\in \cO_{\leq\la}$ is the projective cover of the irreducible module $L_\lambda$.
    \item 
    The costandard module $\nabla_{\lambda}\in \cO_{\leq\la}$ is the injective hull of the same irreducible $L_{\lambda}$.
\item
The proper costandard module $\overline{\nabla}_{\lambda}$ is an indecomposable object of $\cO_{\leq\la}$ defined by the following universal properties:
for each module $M\in \cO_{\leq\la}$ with 
$$[M:L_{\lambda}]=1 \text{ and } [M:L_{\nu}]=0 \text{ for all }\nu\neq\lambda \text{ such that } \rho(\nu)=\rho(\lambda)$$ 
and an embedding $i:L_\lambda\rightarrow M$ there exists a map $\psi_{M}:M\to \overline{\nabla}_{\lambda}$ that commutes with the standard embedding $L_{\lambda}\to \overline\nabla_{\lambda}$. Moreover, for any embedding $M\stackrel{\varphi}\hookrightarrow N$ 
of modules in $\cO_{\leq\la}$ the maps $\psi_M$ and $\psi_N$ may be chosen consistently:
\begin{equation}
\label{eq::proper::costandard}
\begin{tikzcd}
& \overline\nabla_{\lambda} & & \\
0 \arrow[r] & L_{\lambda} \arrow[r,hookrightarrow] \arrow[u,hookrightarrow] & M \arrow[r,hookrightarrow,"\varphi"description] \arrow[ul,"\psi_M"description] & N \arrow[ull,"\psi_N"description]
\end{tikzcd}
\end{equation}
\item
Similarly, one defines the proper standard module $\overline\Delta_{\lambda}$ using the following commutative diagram
\[
\begin{tikzcd}
& & \overline\Delta_{\lambda} \arrow[dll,"\psi_M"description] \arrow[dl,"\psi_N"description] \arrow[d,twoheadrightarrow] & 
\\
M \arrow[r,twoheadrightarrow,"\varphi"description] & N \arrow[r,twoheadrightarrow] & L_{\lambda} \arrow[r] & 0  
\end{tikzcd}
\]
\end{itemize}
\end{dfn}

In other words, the proper costandard module $\overline\nabla_{\lambda}$ (respectively the proper standard module $\overline{\Delta}_{\lambda}$) is an injective hull (resp. projective cover) of the module $L_\lambda$ in the exact subcategory of $\cO_{\leq\la}$ consisting of modules with the multiplicity of the irreducible modules $\{L_\nu\colon \rho(\nu)=\rho(\lambda)\}$ not greater than one.
In particular, one can easily prove the following:
\begin{lem}
\label{lem::ext::1}
\begin{itemize}
    \item 
    If $M\in\cO_{<\lambda}$, then 
    $$\Hom_{\cO}(M,\overline{\nabla}_{\lambda}) = \Hom_{\cO}(\overline{\Delta}_{\lambda},M) = 0.$$
    \item 
    If $\cO$ contains standard objects and proper costandard objects and $\rho(\lambda)$ and $\rho(\mu)$ are comparable elements of $\Theta$, then
    $$\dim\Hom_{\cO}(\Delta_{\lambda},\overline{\nabla}_{\mu}) =\delta_{\lambda,\mu}, \qquad
    \Ext_{\cO}^{1}(\Delta_{\lambda},\overline{\nabla}_{\mu}) =0;$$
\item 
Similarly, if $\cO$ contains proper standard objects and costandard objects, then for comparable $\rho(\lambda)$ and $\rho(\mu)$ we have:
    $$\dim\Hom_{\cO}(\overline\Delta_{\lambda},{\nabla}_{\mu}) =\delta_{\lambda,\mu}, \qquad
    \Ext_{\cO}^{1}(\overline\Delta_{\lambda},{\nabla}_{\mu}) =0.$$
\end{itemize}
\end{lem}
\begin{proof}
 Note that all objects we consider in the statement of the lemma belong to the Serre subcategory $\cO_{\leq \lambda}$ or to the subcategory $\cO_{\leq \max(\lambda,\mu)}$. 
 Thus one can use the properties of (proper) (co)standard modules.
 In particular, the unique nontrivial map from $\Delta_\lambda\to \overline{\nabla}_{\lambda}$ factors through $L_\lambda$:
\begin{equation}
\label{eq::delta::nabla::map}
    \Delta_\lambda \twoheadrightarrow L_{\lambda} \hookrightarrow  \overline{\nabla}_{\lambda}.
\end{equation} 
\end{proof}

\begin{rem}\label{StandardExist}
The (proper) (co)standard object may not exist and is defined up to non-unique isomorphism.
If $\cO$ has enough projectives, then $\Delta_{\lambda}$ does exist and  coincides with the image $\imath_{\leq\la}^{*}(\bP_{\lambda})$ of the projective cover $\bP_{\lambda}$ of $L_\la$ in $\cO$.
Respectively, if $\cO$ has enough injectives, then costandard object $\nabla_{\lambda}$ does exist and coincides  with $\imath_{\leq\la}^{!}(\bI_{\lambda})$, where $\bI_{\lambda}$ is the injective hull of $L_\lambda$ in $\cO$.
\end{rem}

\subsection{Stratified categories}

In what follows, we use the notion of a stratified category \cite{BS, EL, LW}, which generalizes the concept of a highest weight category. 
In representation theory, the first example of a highest weight category was introduced by Bernstein, Gelfand, and Gelfand in~\cite{BGG} to study the category \( \mathcal{O} \) associated with the Cartan decomposition of a semisimple Lie algebra. The first formal definition of a highest weight category, consistent with the BGG construction, was provided by Cline, Parshall, and Scott in \cite{CPS}. 

It is important to note that, in the classical setting of~\cite{CPS}, all the categories under consideration contain a finite number of irreducible objects and possess both enough projectives and enough injectives simultaneously. 
However, the primary focus of this paper is on certain categories of representations of 
infinite-dimensional Lie algebras, which generally have an infinite number of simple objects and 
may have either enough projectives or enough injectives, but not necessarily both.

\subsubsection{Stratified categories with enough projectives}

From now on, we assume that \( \cO^{\star} \) is a \( \Theta \)-ordered category yielding the following assumptions:
\begin{itemize}
\item $\cO^{\star}$ has enough projectives, in particular, each irreducible $L_{\lambda}$ admits a projective cover $\bP_{\lambda}$;
\item the proper costandard objects $\overline{\nabla}_{\lambda}$ belong to $\cO^{\star}$;
\item each object \( M \in \cO^\star \) admits 
a decreasing separated filtration $M=F_0 M\supset F_1 M\supset \ldots$, such that the consecutive subquotients $F_iM/F_{i+1}M$ are simple and each simple appears as a subquotient finite number of times.
\item each object \( M \in \cO^\star \) admits a projective resolution \( P_{\ldot} \twoheadrightarrow M \) such that for each \( \lambda \in \Lambda \) the number of integers
\( k \) for which \( [P_k(M) : L_\lambda] \neq 0 \) is finite. 
\end{itemize}

The corresponding derived category of bounded above complexes, where each irreducible appears only finitely many times, is denoted by \( \mathcal{D}(\cO^\star) \). 
Consequently, the Euler characteristic of a complex is a well-defined map from the derived category to the Grothendieck group:
\[
\chi: \mathcal{D}(\cO^\star) \longrightarrow K_0(\cO^\star), \qquad C^{\udot} \mapsto \sum_{i=-\infty}^\infty (-1)^i [C^i].
\]
Furthermore, for each object \( N \in \cO^\star \) that contains a finite set of irreducibles in its Jordan-H\"older decomposition, we have a well-defined pairing:
\begin{equation}
\label{eq::pairing::def}
([M^{\udot}], [N]) := \sum_{i=-\infty}^\infty (-1)^{i} \dim \RHom^i_{\mathcal{D}}(M^{\udot}, N).
\end{equation}
The result of this pairing depends only on the images of \( M^{\udot} \) and \( N \) in the Grothendieck group \( K_0(\mathcal{O}) \).

\begin{prop}
\label{prp::standard::filtration}
Let $M\in\cO^\star$. 
Under the assumptions of this section, each successive subquotient of the filtration $(M)_{\not\leq \lambda}$~\eqref{eq::filtr::factor} is covered by several copies of the standard modules:
\begin{equation}
\label{eq::stand::filtr}
\bigoplus_{\mu\colon \rho(\mu)=\lambda} \Delta_{\mu}^{m_\mu} \twoheadrightarrow 
\frac{\imath_{\leq \lambda}\imath^{*}_{\leq \lambda} M}{\imath_{<\lambda}\imath^{*}_{<\lambda} M} = \frac{(M)_{\not<\lambda}}{(M)_{\not\leq \lambda}},
\end{equation}
where the multiplicity \( m_{\mu} \) is given by \( \dim \Hom(M,\overline\nabla_{\mu}) \).
\end{prop}

\begin{proof}
It is sufficient to consider the case \( M \in \cO^{\star}_{\leq\lambda} \), \( \imath^{*}_{<\lambda} M=0 \).
Since \( M \) admits a decreasing separated filtration \( M=F_0M\supset F_1M \supset \ldots \) with simple successive subquotients \( L_{\mu_i}=F_iM/F_{i+1}M \) satisfying \( \rho(\mu_i)\leq\rho(\lambda) \), we note that \( \rho(\mu_0)=\rho(\lambda) \). Otherwise, if \( F_0M/F_1M \in \cO^{\star}_{<\lambda} \), then \( \imath_{<\lambda}^{*}M\neq 0 \), leading to a contradiction.

Since the projective cover of \( L_{\mu_0} \) in \( \cO^{\star}_{\leq\lambda} \) is \( \Delta_{\mu_0} \), we obtain a map \( \psi_0:\Delta_{\mu_0}\rightarrow M \), fitting into the commutative diagram:
\[
\begin{tikzcd}
 & \Delta_{\mu_0} \arrow[dl,"\psi_0"description] \arrow[d,twoheadrightarrow] & 
\\
M \arrow[r,twoheadrightarrow,"\varphi"description] &  F_0M/F_1M =L_{\mu_0} \arrow[r] & 0  
\end{tikzcd}
\]
Setting \( M_1:=\coker(\psi_0) \), we obtain:
$$
\begin{tikzcd}
 \Delta_{\mu_0} \arrow[r] \arrow[d,"\imath_{<\lambda}^{*}"] &  M \arrow[r] \arrow[d,"\imath_{<\lambda}^{*}"] & M_1  \arrow[r] \arrow[d,"\imath_{<\lambda}^{*}"] & 0 \\
 0 \arrow[r]& 0 \arrow[r] & \imath_{<\lambda}^{*}(M_1) \arrow[r] & 0
\end{tikzcd} 
 \ \Rightarrow \ \imath_{<\lambda}^{*}M_1=0,
$$
because \( \imath_{<\lambda}^{*} \) is the left adjoint to the exact functor \( \imath_{<\lambda} \), making it right exact.

We now proceed by induction, removing the top irreducible module from \( M_1 \). Since the multiplicity of each irreducible \( L_{\mu} \) in \( M \) is finite, after (possibly infinitely) many steps, we obtain the projective cover \( \psi:\bP(M)\twoheadrightarrow M \). Such that $\bP(M)$ consists of a (possibly infinite) direct sum of modules \( \Delta_{\mu} \) for \( \rho(\mu)=\rho(\lambda) \). However, for each particular \( \mu \), the number of summands is finite.

By construction, for each \( \mu \) with \( \rho(\mu)=\rho(\lambda) \), the map 
\begin{equation}
\label{eq::proj::iso}
\psi^{*}:\Hom(M,\overline{\nabla}_{\mu}) \rightarrow \Hom(\bP(M),\overline{\nabla}_{\mu})
\end{equation}
is surjective. Indeed, each nonzero map from \( \Delta_{\mu} \) to \( \overline{\nabla}_{\mu} \) factors through the irreducible \( L_{\mu} \), as mentioned in~\eqref{eq::delta::nabla::map}. In particular, the nontrivial map from the summand \( \Delta_{\mu_0} \) (created in the first step) to the proper costandard object \( \overline{\nabla}_{\mu_0} \) factors through the irreducible \( L_{\mu_0}=F_0M/F_1M \) (see Lemma~\ref{lem::ext::1} and formula~\eqref{eq::delta::nabla::map}). Thus, this map arises from a map \( \varphi_0\in \Hom(M,\overline{\nabla}_{\mu_0}) \).
Proceeding inductively, we find a preimage (with respect to $\psi^*$) 
of a map from a summand $\Delta_{\mu}\subset \bP(M)$ to $\overline{\nabla}_{\mu}$ in the space $\Hom(M,\overline{\nabla}_{\mu})$.
 Additionally, we observe that \( \psi^{*} \) is injective because \( \RHom(\ttt,\overline{\nabla}_{\mu}) \) is a right-exact functor. Consequently, \( \psi^{*} \) is an isomorphism, and Lemma~\ref{lem::ext::1} completes the proof.
\end{proof}

\begin{dfn}
\label{def::stratified}
\label{item::stratified}
Suppose that the standard object $\Delta_\lambda$ is a projective cover of $L_\lambda$ in all categories $\cO^{\star}_{\le S}$ for $S\ni\la$ such that $\rho(\lambda)$ is the maximal element in $\rho(S)$.   
We say that the category $\cO^{\star}$ is {stratified} iff for any finite subset $S\subset \Lambda$ the derived functor $\imath_{\leq S}: \cD(\cO^{\star}_{\leq S}) \rightarrow \cD(\cO^{\star})$ is a fully faithful embedding.
\end{dfn}

The following Theorem gives a criterion of a stratified category (with projectives) and already appeared in the literature in many different places for the case of highest weight categories
(\cite{BGG},\cite{CPS},\cite{Dlab},\cite{Kh}). However, we would like to add it here to make the presentation more self-contained.

\begin{thm}\label{thm::Stratified}
Suppose that category $\cO^{\star}$ has enough projectives, $\cD(\cO^{\star})$ is generated by projectives and proper costandard objects $\overline\nabla_{\lambda}$ exist in $\cO^{\star}$. Then the following conditions are equivalent:
\begin{enumerate}
    \item[(s1$^\star$)] the category $\cO^{\star}$ is stratified;
   \label{item::D} 
    \item[(s2$^\star$)] $\forall \lambda\in\Lambda$ the kernel of the projection $\imath_{\leq\lambda}^{*}:\bP_\lambda\rightarrow \Delta_{\lambda}$ admits a filtration by $\Delta(\mu)$ with $\rho(\mu)>\rho(\lambda)$;
    \label{item::P}
    \item[(s3$^\star$)] there are no higher derived homomorphisms $\RHom_{\cD(\cO^{\star})}(\Delta_{\lambda},\overline\nabla_{\mu}) = 0$ for all pairs of $\lambda,\mu$ such that $\rho(\lambda)\neq\rho(\mu)$;
    \label{item::Ext}
    \item[(s4$^\star$)] $\forall \lambda,\mu\in \Lambda$ with $\rho(\lambda)\neq\rho(\mu)$ we have  
    $\Ext_{\cD(\cO^{\star})}^2(\Delta_\lambda,\overline\nabla_\mu) = 0$.
    \label{item::Ext2}
\end{enumerate}
\end{thm}
\begin{rem}
Note that a highest weight category defined in~\cite{CPS} is the following simplification of a stratified category:
\begin{itemize}
    \item 
$\Lambda$ coincides with $\Theta$ and is a finite set, 
\item
there are no nontrivial endomorphisms of a standard module $\Delta_\lambda$, 
\item
the standard module $\Delta_\lambda$ coincides with the proper standard module $\overline\Delta_\lambda$, 
\item the derived category $\cD(\cO^{\star})$ is the bounded derived category and is stratified with successive quotient isomorphic to the derived category of vector spaces. 
\end{itemize}
Moreover, \cite{CPS} uses item~$(s2^{\star})$ as a definition of a highest weight category and shows that $\Delta_{\lambda}$ assemble a finite exceptional collection of $\cD(\cO^{\star})$.
Unfortunately, the categories we are dealing with are not of that sort and we have to use the more general notion of a {stratified category}. The notion of a stratified category can be found in more recent papers like \cite{Kh}, \cite{BS} and in \cite{EL}.
\end{rem}
\begin{proof}
We give the proof of the theorem by showing implications $(s2^\star)\Rightarrow (s1^\star)\Rightarrow (s3^\star)\Rightarrow (s4^\star) \Rightarrow (s2^\star).$ 

$(s2^\star)\Rightarrow (s1^\star)$
Let $S$ be a subset of $\Lambda$.
Consider the triangulated category \( \mathcal{D}(\cO^\star)_{\not\leq S} \) as the 
subcategory of \( \mathcal{D}(\cO^\star) \) generated by the set of projectives 
\( \bP_{\mu} \) with $\mu$ satisfying $\rho(\mu) \not\leq \rho(s)$ for all $s\in S$.
The subcategory \( \mathcal{D}(\cO^\star)_{\not\leq S} \) is right 
admissible, and there exists a left orthogonal subcategory 
\( \mathcal{D}(\mathcal{O}^\star)_{\leq S} \subset \mathcal{D}(\cO^{\star})\), consisting of complexes of modules in 
\( \mathcal{O}^\star \), whose homology belongs to the subcategory 
\( \mathcal{O}^\star_{\leq S} \). 
Now, consider the exact sequence of modules:
\[
0 \to K \to \bP_{\mu} \to \imath_{\leq S}^{*}(\bP_{\mu}) \to 0.
\]
By assumption \( (s2^\star) \), the kernel \( K \) admits a filtration with successive quotients isomorphic to \( \Delta(\nu) \), where \( \rho(\nu) \not\leq \rho() \) for all $s\in S$. 
Consequently, \( K \) has a projective resolution by projectives \( P_\nu \) with \( \rho(\nu) \not\leq \rho(s) \) for all $s\in S$. 
Since the modules \( \imath_{S}^{*}(\bP_{\mu}) \) belong to 
\( \mathcal{O}^\star_{\leq S} \), every object in the left orthogonal category \( {}^{\perp} \mathcal{D}(\mathcal{O}^\star)_{\not\leq S} \) admits a resolution by these modules and, therefore, belongs to the derived category \( \mathcal{D}(\mathcal{O}^\star_{\leq S}) \). 
It follows that the triangulated categories \( \mathcal{D}(\mathcal{O}^\star_{\leq S}) \) and 
\( \mathcal{D}(\mathcal{O}^\star)_{\leq S} \) are equivalent, and the embedding 
\( \imath: \mathcal{D}(\mathcal{O}^\star_{\leq S}) \rightarrow \mathcal{D}(\mathcal{O}^\star) \) coincides with the derived functor of the corresponding embedding of abelian categories.

$(s1^\star)\Rightarrow(s3^\star)$
If \( \rho(\lambda) \geq \rho(\mu) \), then both \( \Delta_\lambda \) and \( \overline{\nabla}_{\mu} \) belong to the subcategory \( \mathcal{O}^\star_{\leq\lambda} \). Consequently, by assumption \( (s1^\star) \), we obtain:
\[
\RHom_{\mathcal{D}(\mathcal{O}^\star)}(\Delta_{\lambda},\overline{\nabla}_{\mu}) = 
\RHom_{\mathcal{D}(\mathcal{O}^\star_{\leq\lambda})}(\Delta_{\lambda},\overline{\nabla}_{\mu}) = 
\Hom_{\mathcal{O}^\star_{\leq\lambda}}(\Delta_{\lambda},\overline{\nabla}_{\mu}) = [\overline{\nabla}_{\mu}:L_{\lambda}].
\]
The second equality follows from the fact that \( \Delta_\lambda \) is the projective cover of \( L_\lambda \) in \( \mathcal{O}^\star_{\leq\lambda} \). 
If \( \rho(\lambda) \) and \( \rho(\mu) \) are incomparable, then \( \Delta_{\lambda} \) remains the projective cover of \( L_{\lambda} \) in the subcategory \( \mathcal{O}^\star_{\leq\{\lambda,\mu\}} \) due to Definition~\ref{def::stratified}.
Similarly, for \( \rho(\lambda) \leq \rho(\mu) \), the $\RHom$ groups vanish due to the universal property~\eqref{eq::proper::costandard} of proper costandard modules.

Assumption \( (s4^\star) \) is a special case of \( (s3^\star) \), taking care of 
the vanishing of the second cohomology groups.

$(s4^\star)\Rightarrow(s2^\star)$ 
Consider the standard filtration of \( (\bP_\lambda)_{\not\leq\mu} \) discussed in 
Remark~\ref{remark::filtration}.
We prove by induction on the lower finite set \( \Theta \) that the successive quotients of this filtration admit a further filtration by standard modules.

The base case of the induction is straightforward since \( \imath_{\leq\lambda}^{*}(\bP_\lambda) = \Delta_\lambda \).
Now, assume that \( \imath_{<\lambda}^{*}(\bP_{\lambda}) \) admits a filtration of the desired type. Then, the kernel $M$ of the surjective map
\[
\imath_{\leq\lambda}^{*}(\bP_{\mu})\twoheadrightarrow \imath_{<\lambda}^{*}(\bP_\mu)
\]
is generated by irreducible modules indexed by \( \nu \) with \( \rho(\nu) = \rho(\lambda) \).
Thanks to Proposition~\ref{prp::standard::filtration} there exists a projective cover $P_{\mu}^{=\lambda}\twoheadrightarrow M$ that consists of the direct sum of projective covers $\Delta_{\nu}\twoheadrightarrow L_{\nu}$ with $\rho(\nu)=\rho(\lambda)$.
Consequently, we obtain the following exact sequence of objects of $\cO^{\star}_{\leq\lambda}$:
\begin{equation}
\label{eq::seq::lambda}
0 \to K \to P_{\mu}^{=\lambda} \to \imath_{\leq\lambda}^*(\bP_{\mu}) \twoheadrightarrow \imath_{<\lambda}^*(\bP_\mu) \to 0.
\end{equation}
If \( K \neq 0 \), then there exists a maximal element \( \nu \) such that \( [K:L_{\nu}] \neq 0 \), which implies the existence of a nontrivial homomorphism from \( K \) to \( \overline\nabla_{\nu} \).
Applying \( \RHom(-,\overline\nabla_{\nu}) \) to \eqref{eq::seq::lambda}, we observe that there are no derived homomorphisms from \( P_{\mu}^{=\lambda} \) to 
$\overline\nabla_{\nu}$, and since \( \imath_{<\lambda}^{*}(\bP_{\mu}) \) is an extension of standard modules, there are no higher \(\Ext\)-groups from \( \imath_{<\lambda}^{*}(\bP_{\mu}) \) to \( \overline\nabla_{\nu} \).
We now show that \( \Ext^1(\imath_{\leq\lambda}^*(\bP_{\mu}),\overline\nabla_{\nu}) = 0 \). 
Suppose there exists a short exact sequence
\[
0\to \overline\nabla_{\nu} \to M \to \imath_{\leq\lambda}^*(\bP_{\mu})\to 0.
\]
Since \( M \) belongs to the subcategory \( \cO^{\star}_{\leq\lambda} \), and \( \imath_{\leq\lambda}^*(\bP_{\mu}) \) is projective in \( \cO^{\star}_{\leq\lambda} \), the sequence splits in \( \cO^{\star}_{\leq\lambda} \), and hence also in \( \cO^{\star} \).
This contradicts the assumption that \( \Hom(K,\overline\nabla_{\nu}) \neq 0 \).

It is important to note that we use induction only to add a maximal element to a set. 
On the other hand, if we assume that for a finite set \( S \) and all \( \lambda\in \Theta_{\leq S} \), the standard filtration of \( \imath_{\leq\lambda}^{*}(\bP_\mu) \) admits a refinement with successive 
quotients isomorphic to \( \Delta_{\mu} \), then the standard filtration of \( \imath_{\leq S}^{*}(\bP_{\mu}) \) also has successive quotients isomorphic to \( \Delta_{\mu} \) with \( \rho(\mu) \leq \rho(s) \) for all $s\in S$.
\end{proof}

The following definitions and results are well known to specialists in the highest weight categories and can be found for example in~\cite{Kh}. 

\begin{dfn}
A decreasing filtration $\cF^{\bullet} M$ of a module $M$ in a stratified category $\cO^{\star}$ is called excellent $\Delta$-filtration iff the successive quotients are isomorphic to the direct sums of standard modules.
\end{dfn}

\begin{prop}
If the category $\cO^{\star}$ is stratified, has enough projectives, and contains proper costandard modules, then
a module $M$ admits an excellent $\Delta$-filtration by standard modules if and only if $\Ext^1(M,\overline\nabla_\nu)=0$ for all $\nu\in\Lambda$.
\end{prop}
\begin{proof}
See e.g.~\cite[\S3.5]{Kh}
\end{proof}

\begin{cor}
\label{CorBGG}
Let $\cO^{\star}$ be a stratified category with enough projectives, then each projective $\bP_{\lambda}$ admits an excellent $\Delta$-filtration and the multiplicities of subquotients isomorphic to $\Delta_\mu$ in the projective satisfy the BGG reciprocity
\[
[\bP_\lambda:\Delta_{\mu}] = [\overline\nabla_{\mu}:L_\lambda].
\]
\end{cor}
The following results are very useful for applications.
\begin{prop}
\label{prp::substratified}
For any $\lambda\in\Lambda$ the subcategory $\cO^{\star}_{\leq\lambda}$ of a $\Theta$-stratified category $\cO^{\star}$ is also stratified.
\end{prop}
\begin{proof}
    The straightforward application of Definition~\ref{def::stratified} through the stratification of derived categories.
\end{proof}
\begin{cor}\label{BGGreciprocity}
Suppose that $\Theta$-ordered category $\cO^{\star}$ is stratified and has enough projectives.
Then for any $\lambda\in\Lambda$ the BGG reciprocity holds for any subcategory $\cO^{\star}_{\leq\lambda}$.
In other words, for all $\mu\in\Lambda$
the module $\imath_{\leq\lambda}^{*}\bP_\mu$ admits an excellent filtration by standards $\Delta_{\nu}$ with $\rho(\lambda)\geq\rho(\nu)\geq\rho(\mu)$ and 
$$
[\imath_{\leq\lambda}^*(\bP_\mu): \Delta_\nu] = [\overline\nabla_{\nu}:L_{\mu}].
$$
\end{cor}
\begin{proof}
    Follows from Propositions~\ref{prp::substratified} and Corollary~\ref{CorBGG}.
\end{proof}

\begin{thm}
\label{thm::characters::stratified}
 A $\Theta$-ordered category $\cO^{\star}$ is stratified if the following assumptions are satisfied:
 \begin{enumerate}
     \item $\cO^{\star}$ has enough projectives $\bP_\lambda$ and contains standard modules $\Delta_{\lambda}$ satisfying assumptions from Definition~\ref{def::stratified};
     \item there exists a collection of integers $\{m_{\lambda,\mu}\colon \lambda,\mu\in \Lambda\}$ such that, for all $\la\in\Lambda$, the sum
     $\sum_{\mu} m_{\lambda,\mu} [\Delta_\mu]$ is well defined in $K_0(\cO^{\star})$ and is equal to $[\bP_\lambda]$;
     \item for all $\lambda\in\Lambda$, the proper costandard object $\overline\nabla_{\lambda}$ belongs to $\cO^{\star}$ and its pairing with any $C\in\cD(\cO^{\star})$ is well defined (for example, it is enough to have finite-dimensional $\overline\nabla_\lambda$); 
     \item the characters of standard modules $\Delta_{\lambda}$ and proper costandard modules $\overline\nabla_{\mu}$ constitute a dual basis in the Grothendieck group $K_0(\cO^{\star})$:
\begin{equation}
\label{eq::standard::orthogonal}
 \forall\, \lambda,\mu \in \Lambda:  \qquad 
 ([\Delta_\lambda],[\overline\nabla_\mu]) = \chi(\Ext_{\cD(\cO^{\star})}(\Delta_\lambda,\overline\nabla_{\mu})) = \delta_{\lambda,\mu}.
\end{equation}          
 \end{enumerate}
 \end{thm}
\begin{rem}
Note that condition $(s3^{\star})$ of Theorem~\ref{thm::Stratified} implies that if the category $\cO^{\star}$ is stratified, then standard and proper costandard modules
assemble dual bases in $K_0(\cO^{\star})$, provided 
the pairings between standard and proper costandard modules are well defined.
The surprising fact is that if all summations are well defined, then it is sufficient to have vanishing of the Euler characteristic rather than all $\Ext$-groups.
\end{rem}
\begin{proof}
Let us look more carefully at the pairing with projectives:
\begin{equation}
\label{eq::pairing::with::projective} 
([\bP_{\lambda}],[M]) = \chi(\Ext(\bP_\la,M))= \dim\Hom(\bP_{\lambda},M) = [M:L_\lambda].
\end{equation}
In particular, $[\bP_\lambda]$ and $[L_\mu]$ constitute a dual basis in $K_0(\cO^{\star})$: 
$$
\dim \Ext^{i}_{\cO^{\star}}(\bP_\lambda,L_\mu) = \delta_{\lambda,\mu}\delta_{i,0} \Rightarrow ([\bP_\lambda],[L_\mu])=\delta_{\lambda,\mu}.$$
From assumption $\mathrm{(ii)}$, 
we know that
$$
[\bP_\lambda] = \sum_{\mu\in\Lambda} m_{\lambda,\mu} [\Delta_{\mu}]
$$
holds in $K_0(\cO^{\star})$, and thanks to $\mathrm{(iii)}$ we can take the pairing with any proper costandard module $\overline\nabla_{\mu}$:
$$
[\overline\nabla_{\mu}:L_\lambda] \stackrel{\eqref{eq::pairing::with::projective}}{=} ([\bP_\lambda],[\overline\nabla_\mu]) = \left(\sum_{\nu} m_{\lambda,\nu} [\Delta_\nu], [\overline\nabla_{\mu}]\right) = \sum_{\nu}
m_{\lambda,\nu} \left([\Delta_\nu], [\overline\nabla_{\mu}]\right) = m_{\lambda,\mu}.
$$
The proper standard modules $\overline\nabla_{\mu}$ belong to $\cO^{\star}_{\leq\mu}$ and consequently we have the following equality in $K_0(\cO^{\star})$:
$$
[\overline\nabla_{\mu}] = \sum_{\lambda} [\overline\nabla_{\mu}:L_\lambda] = 
[L_\mu] + \sum_{\lambda\colon \rho(\lambda)<\rho(\mu)} m_{\lambda,\mu} [L_\lambda].
$$
In particular, $m_{\lambda,\mu} = 0$ whenever $\lambda\not\leq\mu$.

Let us consider the standard filtration of the projective cover $\bP_{\lambda}$. Thanks to Proposition~\ref{prp::standard::filtration} we know that the $\mu$-component of the associated graded is covered by the sum of standard modules $\Delta_{\mu}$ with multiplicity $m_{\lambda,\mu} = [\overline\nabla_{\mu}:L_{\lambda}]$. Let us denote the kernel of the corresponding map by $M_{\lambda,\mu}$: $$0\to M_{\lambda,\mu} \rightarrow \Delta_{\mu}^{\oplus m_{\lambda,\mu}}\rightarrow \imath_{\leq\mu}^*(\bP_{\lambda})/\imath_{<\mu}^*(\bP_{\lambda}) \rightarrow 0.$$
Then we have the following equality in the Grothendieck group:
$$
[\bP_{\lambda}] = [\Delta_{\lambda}] + \sum_{\mu\colon \rho(\mu)>\rho(\lambda)} (m_{\lambda,\mu}[\Delta_{\mu}] -[M_{\la,\mu}]).
$$
It follows that
$$
\sum_{\mu\colon \rho(\mu)>\rho(\lambda)} [M_{\la,\mu}] = 0,\quad \text{and hence} \quad M_{\lambda,\mu} = 0 \quad \text{for all $\la,\mu.$}
$$
We deduce that $\bP_\lambda$ admits a filtration by standard modules $\Delta_{\mu}$ with $\rho(\mu)\geq\rho(\lambda).$
\end{proof}

\subsubsection{Stratified categories with enough injectives}
Let us outline similar results for the case of ordered categories with enough injectives. We do not include the proofs, as they are identical to those in the preceding section, with the directions of the arrows reversed.

In this section we assume that \( \cO \) is a \( \Theta \)-ordered category satisfying the following assumptions:
\begin{itemize}
\item $\cO$ has enough injectives, in particular, each irreducible $L_{\lambda}$ admits an injective hull $\bI_{\lambda}$;
\item the costandard object $\nabla_{\lambda}$ is an injective hull of $L_\lambda$ in all categories $\cO_{\le S}$ for $S\ni\la$ such that $\rho(\lambda)$ is the maximal element in $\rho(S)$;
\item the proper standard objects $\overline{\Delta}_{\lambda}$ belongs to $\cO$;
\item each object \( M \in \mathcal{O} \) admits 
an increasing exhaustive filtration $0=F_0 M\subset F_1 M\subset \ldots$, such that the consecutive subquotients $F_iM/F_{i-1}M$ are simple and each simple appears as a subquotient finite number of times;
\item each object \( M \in \mathcal{O} \) admits an injective  resolution \(M\hookrightarrow \bI_{\ldot} \) such that, for each \( \lambda \in \Lambda \), the number of integers \( k \) for which \( [\bI_k(M) : L_\lambda] \neq 0 \) is finite. 
\end{itemize}
We denote by $\cD(\cO)$ the corresponding derived category consisting of bounded below complexes, where each irreducible appears only a finite number of times.

\begin{prop}
For each module $M$ there exist embeddings of the subquotients of the filtration~\eqref{eq::filtr::sub} into the sum of several copies of the costandard modules:
\[
\imath_{\leq\lambda}\imath^{!}_{\leq\lambda} M / \imath_{<\lambda}\imath^{!}_{<\lambda} M \hookrightarrow \oplus_{\mu\colon \rho(\mu)=\lambda} \nabla_{\mu}^{\oplus m_{\mu}}
\text{ with } m_{\mu}:=\dim\Hom(\overline\Delta_{\mu},M).
\]
\end{prop}

\begin{dfn}
We say that the category $\cO$ is stratified iff for any finite subset $S\subset \Lambda$ the derived functor $\imath_{\leq S}: \cD(\cO_{\leq S}) \rightarrow \cD(\cO)$ is a fully faithful embedding.
\end{dfn}

\begin{thm}
Suppose that category $\cO$ has enough injectives, proper standards $\overline\Delta_\lambda$ exist in $\cO$ and $\cD(\cO)$ is generated by injectives. Then the following conditions are equivalent:
\begin{enumerate}
    \item[(s1)] the derived category $\cD(\cO)$ is stratified;
    \item[(s2)] $\forall \lambda\in\Lambda$ the cokernel of the injection $\imath_{\lambda}:\nabla_{\lambda}\rightarrow \bI_{\lambda}$ admits a filtration by $\nabla(\mu)$ with $\rho(\mu)>\rho(\lambda)$;
    \item[(s3)] there are no derived homomorphisms $\RHom(\overline\Delta_{\lambda},\nabla_{\mu})$ for all pairs of $\lambda,\mu$ such that $\rho(\lambda)\neq\rho(\mu)$;
    \item[(s4)] $\forall \lambda,\mu$ with $\rho(\lambda)\neq\rho(\mu)$ we have  $ \Ext_{\cD(\cO)}^2(\overline\Delta_\lambda,\nabla_\mu) = 0$.
\end{enumerate}    
\end{thm}

\begin{dfn}
 An increasing filtration $\cF^{\bullet} M$ of a module $M$ in the stratified category $\cO$ is called excellent $\nabla$-filtration iff the successive quotients are isomorphic to the direct sum of costandard modules. 
\end{dfn}

\begin{prop}
Suppose the category $\cO$ is stratified and has enough injectives and proper standard. 
Then 
\begin{itemize}
    \item the module $M$ admits an excellent $\nabla$-filtration by costandard modules if and only if 
    $$\forall \mu\in \Lambda \ \Ext^1(\overline\Delta_{\mu},M)=0;$$
\item one has BGG reciprocity 
\[
[\bI_\lambda:\nabla_{\mu}] = [\overline\Delta_{\mu}:L_\lambda].
\]
for the multiplicities of costandards in the excellent $\nabla$-filtration of injectives;
\item for any $\lambda\in\Lambda$ the subcategory $\cO_{\leq\lambda}$ is also stratified;
in particular, the module $\imath_{\leq\lambda}^{!}\bI_\mu$ admits an excellent filtration by costandards $\nabla_{\nu}$ with $\rho(\lambda)\geq\rho(\nu)\geq\rho(\mu)$ and we have equality of multiplicities:
$$
[\imath_{\leq\lambda}^!(\bI_\mu): \nabla_\nu] = [\overline\Delta_{\nu}:L_{\mu}].
$$
\end{itemize}
\end{prop}

\begin{thm} \label{thm:strat-inj}
The following conditions are sufficient for a $\Theta$-ordered category $\cO$ with enough injectives to be stratified:

Proper standard modules $\overline\Delta_{\lambda}$ and costandard modules $\nabla_\mu$ are well defined and assemble dual bases in the Grothedieck group $K_0(\cO)$; the characters of injectives $[\bI_{\lambda}]$ can be written as linear combinations of characters of costandard modules $[\nabla_{\mu}]$ in the Grothendieck group.
\end{thm}

\section{Structure of bimodule of functions}\label{BimoduleFunctions}
In this section, we study the bi-module of functions on an algebraic group $G$ satisfying certain technical assumptions. We use the categorical approach developed in Section \ref{HWC}.

\subsection{Reductive Lie algebras}
We collect the main notation and definitions of reductive Lie algebras to be used in this section. Let $\fg_0=\fg_0^{ss}\oplus \fg_0^{a}$ be the 
decomposition of a reductive Lie algebra into semi-simple and abelian parts. All Lie algebras are defined over a fixed algebraically closed field $\Bbbk$ of characteristic $0$.

Let us consider the Cartan decomposition $\fg_0^{ss}=\fn_+\oplus\fh\oplus\fn_-$. Let
$\Delta=\Delta_+\sqcup \Delta_-\subset\fh^*$ be the root system of $\fg_0^{ss}$. 
For each $\al\in\Delta_+$, we denote by $e_\al\in\fn_+$ the corresponding Chevalley generator. Similarly, for
$\al\in\Delta_-$, we denote by $f_\al$ the Chevalley generator of weight $\al$ in $\fn_-$.
For $\alpha\in \Delta$, we write $\alpha>0$ for $\alpha\in\Delta_+$ and $\alpha<0$ for $\alpha\in\Delta_-$. We denote by $\{ \al_i \}_{i \in I}$ the simple roots and by
$\{\om_i\}_{i \in I}$ the fundamental weights. For a root $\al\in\Delta$, we denote by $\al^\vee\in\fh$ the corresponding coroot.
For the canonical pairing $\bra\cdot,\cdot\ket:\fh^*\times\fh\to\bC$, we have $\bra \om_i,\al_j^\vee \ket
=\delta_{ij}$. 

Let $Q(\fg_0^{ss})=\bigoplus_{i\in I}\ZZ\alpha_i$ be the root lattice and $Q_+(\fg_0^{ss})=\bigoplus_{i\in I}\ZZ_{\ge 0}\alpha_i$ the positive root cone. 
Let $P(\fg_0^{ss})=\bigoplus_{i\in I}\bZ\om_i$ be the weight lattice and $P_+(\fg_0^{ss})=\sum_{i\in I} \bZ_{\ge 0} \om_i$ the dominant weight cone. Let $\rho=\frac{1}{2}\sum_{\alpha\in\Delta_+} \alpha\in P_+$.

Now let us fix a basis $\beta_1,\dots,\beta_r\in (\fg_0^{a})^*$ of the dual space to the abelian subalgebra. Then we define
\[
P(\fg_0)=P(\fg_0^{ss})\oplus \bigoplus_{i=1}^r \bZ\beta_i,\ P_+(\fg_0)=P_+(\fg_0^{ss})\oplus \bigoplus_{i=1}^r \bZ\beta_i
\]
(note that $P_+(\fg_0)$ contains the whole sublattice $\bigoplus_{i=1}^r \bZ\beta_i$). 

Let $W=W(\fg_0)=W(\fg_0^{ss})$ be the Weyl group with the longest element $w_0$.
For a root $\al\in\Delta$, the corresponding reflection is denoted by $s_\al\in W$ and for each $i \in I$, we set $s_{i} := s_{\al_{i}}$.

For $P=P(\fg_0)$, let $\ZZ[P]$ be the group algebra of $P$, which is a free $\ZZ$-module with basis $\{x^\lambda\}_{\lambda\in P}$ satisfying $x^\lambda x^\mu=x^{\lambda+\mu}$ for all $\lambda,\mu\in P$. Denote by $\ZZ[[P]]$ be the $\ZZ[P]$-module consisting of all (possibly infinite) formal sums $\sum_{\lambda\in P} c_\lambda x^\lambda$ where $c_\lambda\in\ZZ$. Setting $x_i=x^{\omega_i}$ for $i\in I$, we can also write $x^\lambda=\prod_{i\in I}x_i^{\langle\lambda,\alpha_i^\vee\rangle}$. More generally, for any commutative ring $R$ with $1$, we define $R[P]=R\otimes\ZZ[P]$ and $R[[P]]=R\otimes\ZZ[[P]]$.

For any (left) $\fh$-module $M$ and $\lambda\in\fh^*$, let ${}_\lambda M=\{m\in M : \text{$hm=\lambda(h)m$ for all $h\in \fh$}\}$ be the corresponding weight space. If $M$ has a weight space decomposition
$M=\displaystyle\bigoplus_{\lambda \in P}{}_\lambda M$ such that $\dim {}_\lambda M<\infty$ for all $\la\in P$,
then we define the character
$\ch\, M=\sum_{\lambda \in P}(\dim {}_\lambda M) x^\lambda\in\ZZ[[P]].$

Analogously, for a right $\fh$-module $M$ and $\la\in \fh^*$, we define the weight spaces $M_\lambda=\{m\in M : \text{$mh=\lambda(h)m$ for all $h\in \fh$}\}.$ If $M$ has a decomposition
$M=\displaystyle\bigoplus_{\lambda \in P}M_\lambda$ such that $\dim M_\lambda <\infty$ for all $\la\in P$,
then we define the character
$\ch\, M=\sum_{\lambda \in P}(\dim M_\lambda) y^\lambda$
inside another copy of $\ZZ[[P]]$ with basis elements denoted $\{y^\lambda\}_{\lambda\in P}$.

Finally, for $M$ an $(\fh,\fh)$-bimodule, we set ${}_\lambda M_\mu={}_\lambda M \cap M_\mu$ for all $\lambda,\mu\in \fh^*$. If $M$ has a decomposition
$M=\displaystyle\bigoplus_{\lambda,\mu\in P} {}_\lambda M_\mu, \dim {}_\lambda M_\mu<\infty$,
then we define the bi-character
$\ch\, M=\sum_{\lambda,\mu \in P}(\dim {}_\lambda M_\mu) x^\lambda y^\mu$
as an element of $\ZZ[[P\times P]]=\{\sum_{\lambda,\mu} c_{\lambda,\mu} x^\lambda y^\mu : c_{\lambda,\mu}\in\ZZ\}$. The latter is a module over $\ZZ[P\times P]\cong\ZZ[P]\otimes\ZZ[P]$, whose elements are finite sums of the above form. 

Later (see \S\ref{curalg}), we also consider the abelian group $(\ZZ[[P\times P]])[[q]]$ consisting of power series in $q$ with coefficients in $\ZZ[[P\times P]]$. This is a module over the ring $(\ZZ[P\times P])[[q]]$.

\subsection{Pro-algebraic groups}
\label{progroup}
In this subsection, we describe the assumptions on a nonreductive Lie algebra $\fg$ (whose Lie group is proalgebraic and possibly infinite-dimensional); the categories of representation of such algebras are studied below. In particular, 
Theorem~\ref{FT} states that if these assumptions are satisfied and the category of representations is stratified, then the algebra of functions on the corresponding group admits a filtration, which provides a generalization of the Peter-Weyl theorem.

Let $\fg$ be a $\bZ_{\ge 0}$ graded Lie algebra, $\fg=\bigoplus_{m\ge 0} \fg(m)$. We assume the following:
\begin{itemize}
    \item $\fg=\fg_0\oplus \mathfrak{r}$, where $\fg_0\subset \fg(0)$ is a reductive Lie subalgebra with a Cartan subalgebra $\fh\subset\fg_0$,
    \item $\mathfrak{r}=\bigoplus_{m\ge 0} \mathfrak{r}(m)$ is a graded ideal,
    \item each graded component $\mathfrak{r}(m)$, $m>0$, is a finite-dimensional $\fg_0$-module,
    \item $\dim \mathfrak{r}(0)<\infty$ and all the  $\fh$-weight spaces of $\U(\mathfrak{r}(0))$ are finite-dimensional;
    \item $\bigcap_{i=0}^\infty (ad\, \mathfrak{r})^i.\fg =0$
    \item $\dim(\fg/(ad\, \mathfrak{r})^i.\fg)<\infty$ for any $i \geq 0$.
\end{itemize}

We assume that there exists a group $G$ such that 
\begin{itemize}
\item $G=G_0\ltimes R$, where $G_0$ is connected reductive finite-dimensional group and $R$ is a normal subgroup,
\item The sequence $R=R_0\supset R_1\supset\dots$ defined by $R_{i+1}= [R,R_i]$ satisfies  $\bigcap_{i=0}^\infty R_i=\{e\}$, where $e\in G$ is the unit element.
\item ${\rm Lie}(G_0) = \fg_0$,
${\rm Lie}(G/R_i)=\fg/(ad\, \mathfrak{r})^i.\fg$.
\item for each $i\ge 0$, $R_i/R_{i+1}\simeq \mathbb{G}_a^{M}$ for some $M\ge 0$ (depending on $i$), where $\mathbb{G}_a=(\Bbbk,+)$.
\end{itemize}

Let $\overline{\fg}\subset \fg$ be the $\fh$-weight zero subalgebra of $\fg$ (in particular, $\fg_0\supset\fh\subset\overline{\fg}$). 
In what follows we assume that
\begin{equation}\label{abelian}
\overline{\fg} \text{ is abelian Lie algebra}.
\end{equation}

We note that all the assumptions above except for assumption~\eqref{abelian} are satisfied for any finite-dimensional Lie algebra.

\begin{example}
The following Lie algebras satisfy all the  conditions above:
\begin{itemize}
    \item Borel and parabolic subalgebras of semisimple Lie algebras;
    \item current Lie algebras (untwisted and twisted);
    \item Iwahori Lie algebras.
\end{itemize}
\end{example}

\begin{rem}\label{rem:Iwahori}
Recall that the Iwahori algebra attached to a simple Lie algebra $\mathfrak{s}$ is defined as $\mathcal{I}=\fb\oplus \mathfrak{s}\T z\Bbbk[z]\subset \mathfrak{s}[z]$,
where $\fb\subset\mathfrak{s}$ is a Borel subalgebra.
Let $S$ be the connected simply-connected Lie group of $\mathfrak{s}$. The (pro-algebraic) Iwahori group $\bfI\subset S[[z]]$ is defined as the preimage of the Borel subgroup 
$B\subset S$ under the $z=0$ evaluation map $S[[z]]\to S$. Then  the pair  
$\fg=\mathcal{I}$, $G=\bfI$ satisfies the conditions above (see the proof of Theorem \ref{TheoremStr} for more details). We note that the Lie algebra of $\bfI$ is the completion $\fb\oplus \mathfrak{s}\T z\Bbbk[[z]]\subset \mathfrak{s}[[z]]$ of 
$\mathcal{I}$, not $\mathcal{I}$ itself, however the use of $\mathcal{I}$ allows to avoid infinite linear combinations.
\end{rem}

We introduce two categories $\fC$ and $\fC^\star$  of $\fg$ modules:
\begin{itemize}
    \item 
Let 
 $\fC$ be the category of $\bZ$-graded $\fg$ modules $M=\bigoplus_{m\in\bZ} M(m)$ subject to the following conditions: 
\begin{itemize}
\item $M(m)=0$ for $m$ large enough $m>>0$,
\item each $M(m)$ is a direct sum of finite-dimensional $G_0$-integrable irreducible $\fg_0$-modules with finite multiplicities,
\item all the $\fh$-weight spaces of $\ker \mathfrak{r}^i$ are finite-dimensional for all $i> 0$,
\item for any $v\in M$ there exists an $i>0$ such that $\mathfrak{r}^iv=0$, i.e. $\bigcup_{i=1}^\infty \ker \mathfrak{r}^i=M$.
\end{itemize}
The morphisms in $\fC$ are the graded homomorphisms of $\fg$ modules.
\item
The "opposite" category $\fC^\star$ consists of modules $\bZ$-graded $\fg$ modules $M=\bigoplus_{m\in\bZ} M(m)$ such that 
\begin{itemize}
\item $M(m)=0$ for $m<<0$,
\item each $M(m)$ is a direct sum of finite-dimensional $G_0$-integrable irreducible $\fg_0$ modules with finite multiplicities,
\item $M$ is generated as an $\mathfrak{r}$ module by a subspace $\overline{M}$ admitting a decomposition into integrable irreducible $\fg_0$ modules with finite multiplicities,
\item all the $\fh$-weight spaces of $M/\mathfrak{r}^i M$ are finite-dimensional for all $i>0$.
\end{itemize}
The morphisms in $\fC^\star$ are the graded homomorphisms of $\fg$ modules.
\end{itemize}
For a module $M$  from $\fC$ or from $\fC^\star$ we denote by $M^\vee$ the restricted dual module, i.e. the direct sum over $m$ of duals of $\fg_0$ modules showing up as summands of $M(m)$.

We {assume} that there exists a partial order 
$\preceq$ on $P_+(\fg_0)$ yielding the following properties:
\begin{itemize}\label{AssumptionsOrder}
\item[(o1)] For any $\la,\mu\in P_+(\fg_0)$ such that $\la-\mu$ is a sum of positive roots for $\fg_0$ one has $\la\succeq \mu$.
\item[(o2)] Let $\gamma_1,\dots,\gamma_s$ be (nonzero) roots of $\fg$ with respect to the Cartan subalgebra $\fh\subset\fg_0$. Assume that $\sum_{i=1}^s \gamma_i=0$. Then there exists a subset $S\subset \{1,\dots,s\}$ such that 
\begin{equation}\label{assumption}
\la+\sum_{i\in S}  \gamma_i \not\preceq \la.
\end{equation} 
\end{itemize}

The irreducible objects in categories $\fC$ and $\fC^{\star}$ are irreducible $\fg_{0}$-modules placed in appropriate grading, so they are indexed by $\Lambda:=P_{+}(\fg_0)\times \bZ$.
The projection $\rho:\Lambda\rightarrow P_{+}(\fg_0)$ on the first factor defines an order on the abelian categories $\fC$ and $\fC^{\star}$ in the sense of Definition~\ref{Ordered:Category}. So we can apply the machinery of ordered and stratified categories from Section~\ref{HWC} to the categories $\fC$ and $\fC^{\star}$.
In particular, Lemma~\ref{lem::standard}  describes the structure of standard objects $\Delta_{\lambda}$ that belong to $\fC^{\star}$, and costandard objects $\nabla_{\lambda}$ belonging to $\fC$. 

The final assumption we make is as follows:
\begin{equation}\label{Alambda}
{\rm End}_{\fC}(\nabla_\lambda) \simeq {\rm End}_{\fC^{\star}}(\Delta_{-w_0(\lambda)}).
\end{equation}

\begin{example}
The following Lie algebras and their appropriate categories of representations satisfy the conditions above:
    \begin{itemize}
        \item Borel subalgebra (\cite{vdK}, Remark \ref{VanDerKallen});
        \item current Lie algebras with dominance order \cite{Kh};
        \item Iwahori Lie algebras with Cherednik order (Theorem \ref{TheoremStr}).
    \end{itemize}
\end{example}

\subsection{Algebras of functions on pro-algebraic groups}\label{sec::functions}

We study the bimodule $\Bbbk[G]$ of functions on $G$. More precisely, we realize
$G$ as the projective limit of finite-dimensional algebraic groups
\[
G/R_1\twoheadleftarrow G/R_2 \twoheadleftarrow G/R_3\twoheadleftarrow \dotsm
\]
The corresponding spaces of functions form a sequence
\[
\Bbbk[G/R_1]\hookrightarrow \Bbbk[G/R_2] \hookrightarrow \Bbbk[G/R_3]\hookrightarrow \dotsm
\]
and we define $\Bbbk[G]$ as the inductive limit of this sequence.  Note that one has a natural embedding $\Bbbk[G_0]\hookrightarrow\Bbbk[G]$, since $G/R_1\simeq G_0$.

\begin{lem}
 Each element of $\Bbbk[G]$ can be evaluated at an element from $G$. The subspace $\Bbbk[G/R_k]$ consists of elements $f$ such that $f(gr)=f(g)$ for all $r\in R_k$.    
\end{lem}

We note that the space of functions $\Bbbk[G]$ admits commuting left and right actions of the Lie algebra $\fg$ thanks to the realization of $\Bbbk[G]$ as an inductive limit $\Bbbk[G/R_i]$ and the left-right group action (the actions of $\fg$ are induced by the left and right actions of the group $G$ on itself). 
In what follows we study this bimodule structure of $\Bbbk[G]$.

\begin{prop}
For each nonzero $f \in\Bbbk[G]$ there exists an element $x \in \U(\fr)$ such that $x.f \in \Bbbk[G_0]$ is a nonzero element. Similarly, for each nonzero $f \in\Bbbk[G]$ there exists an element $x \in \U(\fr)$ such that $f.x \in \Bbbk[G_0]$ is a nonzero element.  
\end{prop}
\begin{proof}
Assume that $f\in \Bbbk[G/R_k]$. We prove that:
\begin{itemize}
 \item the space $\U(\fr)f$ contains an element which is invariant with respect to $R$,
 \item $\Bbbk[G]^R=\Bbbk[G_0]$.
\end{itemize}
To prove the first claim we note that the group $R_k$ stabilizes $\Bbbk[G/R_k]$. Hence $\Bbbk[G/R_k]$ is naturally a module over the Lie algebra $\fr/(ad\,\mathfrak{r})^k.\fg \subset  \fg/(ad\, \mathfrak{r})^k.\fg$. Since each element from $\fr/(ad\, \mathfrak{r})^k.\fg$ acts on $\Bbbk[G/R_k]$ locally nilpotently, the dimension of $\U(\fr)f$ is finite.
Therefore, this space (as a finite-dimensional representation of the nilpotent Lie algebra $\fr/(ad\, \mathfrak{r})^k.\fg$) contains a nonzero vector which is killed by the whole Lie algebra. 

Now let us show that $\Bbbk[G]^R=\Bbbk[G_0]$. It suffices to show that $\Bbbk[G/R_k]^R=\Bbbk[G_0]$ for all $k$, and this is easy to check.
\end{proof}

\begin{rem}
 The proof of the proposition above implies that for any $f\in\Bbbk[G]$ all the weights of the space $\U(\fg)f$ are bounded from above.  
\end{rem}

In the following proposition we use the restricted dual space for the universal enveloping algebra $\U(\mathfrak{r})$. Explicitly,
the grading $\mathfrak{r}=\bigoplus_{m\ge 0} \mathfrak{r}(m)$ induces the grading on the universal enveloping algebra of $\mathfrak{r}$ such
that each homogeneous component is finite-dimensional (see the beginning of section \ref{progroup}). 
Then $\U(\mathfrak{r})^\vee$ is the direct sum of dual spaces to all homogeneous components.

\begin{prop}\label{Gfunc}
One has an isomorphism of left $\fg$-modules
\[
\Bbbk[G]={\rm coInd}_{\fg_0}^{\fg}\Bbbk[G_0]=\Bbbk[G_0]\otimes \U(\mathfrak{r})^\vee\]
where $\Bbbk[G_0]$ has the trivial action of $\mathfrak{r}$. In particular, $\Bbbk[G]$ belongs to the category $\fC$.
\end{prop}
\begin{proof}
 By definition, $\Bbbk[G]=\varinjlim_{k}\Bbbk[G/R_k]$ (the inductive limit). Now 
 $G/R_k\simeq G_0\ltimes R/R_k$ and one has the tensor product decomposition
 $\Bbbk[G/R_k]\simeq \Bbbk[G_0]\otimes \Bbbk[R/R_k]$. Now we derive the proposition taking the $k\to\infty$ limit. 
\end{proof}

 Finite-dimensional irreducible representations of the Lie algebras $\fg_0$ are labeled by the elements $\la\in P_+(\fg_0)$, where $P_+(\fg_0)$ is the set of dominant weights of $\fg_0$ (we note that  $P_+$
 is the sum of the cone of dominant weights of the semisimple part of $\fg_0$ and the weight lattice of the abelian part).
For $\la\in P_+$  let $L_\la$ be the corresponding  irreducible $\fg_0$ module. We denote by the same symbol the irreducible $\fg$ module obtained by extending trivially the $\fg_0$ action to the whole $\fg$.
We note that, up to degree shifts, the $\fg$ modules $L_\la$ exhaust the irreducible objects in $\fC$ and $\fC^\star$. For $(\la,k)\in P_+\times\ZZ$, let $L_{(\la,k)}$ denote the shift of $L_\la$ which is nontrivial in degree $k$. In general, here and below, objects indexed by $\la\in P_+$ are understood to be the same as those indexed by $(\la,0)\in P_+\times\ZZ$.

We identify $K_0(\fC)$ and $K_0(\fC^\star)$ with $\ZZ[[P]][[q]]$ by sending $[L_{(\la,k)}]$ to $q^k\ch\,L_\la$. Then $[M]$ is sent to $\ch\,M$.

\begin{prop}
\begin{itemize}
\item 
The category $\fC$ contains enough injective objects and contains no projective objects. 
The injective hull of the irreducible module $L_\lambda$ is given by
\[\mathbb{I}_\lambda=\mathrm{coInd}_{\fg_0}^{\fg}L_\lambda=L_\lambda\otimes \U(\fr)^\vee.\]

\item
In the same way $\fC^\star$ contains enough projective objects and contains no injective objects. 
The projective cover of the irreducible module $L_\lambda$ is given by
\[\mathbb{P}_\lambda=\mathrm{Ind}_{\fg_0}^{\fg}L_\lambda=L_\lambda\otimes \U(\fr).\]
\end{itemize}
\end{prop}
\begin{proof}
 The proof is a standard trick in homological algebra.
 The induction functor is left adjoint to the restriction functor which is an exact functor. Consequently, induction maps projectives to projectives.
 Respectively, the conduction is right adjoint to the restriction functor and, therefore, maps injectives to injectives.
\end{proof}

Assume that $M$ is a left $\fg$-module. Recall the definition of the corresponding right module $M^o$. It is isomorphic to $M$ as a vector space and the right action of $\fg$ is the negated left action, i.e.:
\[m x:=- x m, ~ x \in \fg, ~m \in M.\]
We define the categories $\fC^o$, $\fC^{o\star}$ of right $\fg$-modules in the same way as we did for $\fC$, $\fC^\star$ (the only difference is that we consider right modules instead of left modules). Namely, $\fC^{o}$ consists of right $\fg$-modules with bounded from above grading and $\fC^{o\star}$ consists of right bounded from below graded $\fg$-modules.

For an irreducible $\fg$ representation $L_\lambda$, $\la\in P_+(\fg_0)$ with the highest weight $\lambda$ the right module $L_\lambda^o$ contains a vector of right weight $-w_0(\lambda)$, where $w_0$ is the longest element of  $\fg_0$ Weyl group.

\begin{prop}\label{moduleFunctions Decomposition}
    The bimodule $\Bbbk[G]$ is injective left $\fg$-module. Moreover, it has the following decomposition
    \[\Bbbk[G]=\bigoplus_{\la \in P_+(\fg_0)}\mathbb{I}_\la \otimes (L_{-w_0\la})^o\]
    as a $\fg\ttt\fg_{0}$-bimodule.
\end{prop}
\begin{proof}
Proposition \ref{Gfunc} says that 
$\Bbbk[G]\simeq \Bbbk[G_0]\otimes \U(\mathfrak{r})^\vee$, where 
$\U(\mathfrak{r})^\vee$ is the restricted dual space.
Now thanks to the Peter-Weyl theorem one has 
\[
\Bbbk[G_0]
=\bigoplus_{\la\in P_+(\fg_0)} L_\la\otimes (L_{-w_0\la})^{o}.
\]
Now the desired equality follows from the isomorphism  
$\mathbb{I}_\la\simeq 
\U(\mathfrak{r})^\vee\T L_\la$.
\end{proof}


Recall that $M^\vee$ is the restricted dual module.
\begin{lem}
\label{lem::standard}
The map 
$M\mapsto (M^{\vee})^{o}$ sends left modules to right modules. It   
interchanges projective and injective modules, and  (proper) standard and (proper) costandard modules. 
\end{lem}
\begin{proof}
For a left module $M$ its restricted dual $M^\vee$ is also a left module.
Hence the functor $M\mapsto (M^{\vee})^{o}$ is the standard duality between the left and right modules for the universal enveloping algebra;
in particular, the projective and injective modules, and the standard and costandard modules are interchanged.   
\end{proof}


\subsection{Highest weight algebras}
Recall the Lie algebra  $\overline{\fg}\subset \fg$, which is assumed to be abelian  \eqref{abelian}.
Let
\[
\mathcal{A}_\lambda = {\rm End}_{\fC}(\nabla_\lambda),\qquad 
\widetilde{\mathcal{A}}_\lambda = 
{\rm End}_{\fC^\star}(\Delta_{-w_0(\lambda)}) =
{\rm End}_{(\fC^\star)^{o}}((\Delta_\lambda^\vee)^o). 
\]
We note that $\Delta_\la$ being a quotient of the projective cover of the irreducible module is cyclic as a $\fg$-module. Similarly,   $\nabla_\la$ is a cocylcic $\fg$-module. We also note that by assumptions ${\rm (o1), (o2)}$ on p.\pageref{AssumptionsOrder}, 
both $\mathcal{A}_\lambda$ and $\widetilde{\mathcal{A}}_\lambda$ are quotients of the universal enveloping algebra $\U(\overline\fg)$ for all $\la\in P_+(\fg_0)$; in particular, both endomorphism algebras are commutative and graded (since we consider only graded morphisms). We denote the kernels of the corresponding surjections by $K_\la$ and $\widetilde{K}_\la$. According to assumption \eqref{Alambda}
\begin{equation}\label{kernel}
K_\la = \widetilde{K}_\la, \qquad \text{ implying }   \qquad
{\rm End}(\nabla_\lambda)\simeq {\rm End}(\Delta_\lambda^\vee)^o.
\end{equation}
We denote by $\mathcal{A}_\lambda$ both endomorphism algebras 
${\rm End}(\nabla_\lambda) \simeq {\rm End}(\Delta_\lambda^\vee)^o$.   

The proof of the following lemma is standard, but we include it for the reader's convenience.

\begin{lem}
The quotient category $\fC^\star_{=\lambda}:=\fC^\star_{\leq\lambda}/\fC^\star_{<\lambda}$ is isomorphic to the category of graded bounded from below $\cA_{\lambda}$-modules and the restriction on the weight $\lambda$-subspace defines a projection functor $r_{\lambda}:\fC^\star_{\leq\lambda} \rightarrow \cA_\la-\rm{mod}$, whose left adjoint  $r_{\lambda}^*$ coincides with $\Delta_{\lambda}\otimes_{\cA_\la} \texttt{-}$.

The quotient category $\fC_{=\lambda}:=\fC_{\leq\lambda}/\fC_{<\lambda}$ is isomorphic to the category of graded bounded from above $\cA_{\lambda}$-modules and the restriction on the weight $\lambda$-subspace defines a projection functor $r_{\lambda}:\fC_{\leq\lambda} \rightarrow \cA_\la-\rm{mod}$, whose right adjoint $r_{\lambda}^!$ coincides with $\Hom_{\cA_{\la}}(\nabla_{\la},\texttt{-})$.
\end{lem}
\begin{proof}
We note that $\ker (r_\la) = \fC_{<\lambda}=\rm{Im} (\imath_\la)$. Hence it suffices to identify the left and right  adjoint functors as in the statement of our Lemma.

For that, one has  to check that for any $B\in\cA_\la-\rm{mod}$ and $M\in \fC_{\leq\lambda}$ there exist natural isomorphisms
    \[\Hom_{\cA_\la}(B,r_{\lambda}(M))\simeq \Hom_{\fC^\star_{\leq\lambda}}(\Delta_{\lambda}\otimes_{\cA_\la}B, M);\]
    \[\Hom_{\cA_\la}(r_{\lambda}(M),B)\simeq \Hom_{\fC_{\leq\lambda}}(M,\Hom_{\cA_{\la}}(\nabla_{\la},B)).\]
Let us check the first equality, the proof of the second one is similar.
First we have that the $\lambda$ weight component of $\Delta_{\lambda}$ is isomorphic to the free cyclic module over $\mathcal{A}_\lambda$.
Therefore $r_{\lambda}(\Delta_{\lambda}\otimes_{\cA_\la}B)\simeq B$. We have the natural map
\[\Hom_{\fC^\star_{\leq\lambda}}(N, M)\rightarrow \Hom_{\cA_\la} (r_{\lambda}N, r_{\lambda}M).\]
Thus we have a natural map
\[
\tau:\Hom_{\fC^\star_{\leq\lambda}}(\Delta_{\lambda}\otimes_{\cA_\la}B, M)\rightarrow \Hom_{\cA_\la}(B,r_{\lambda}(M)).
\]

  We have the natural morphism
  \[\Hom_{\cA_\la}(B,C)\rightarrow \Hom_{\fC^\star_{\leq\lambda}}(\Delta_{\lambda}\otimes_{\cA_\la}B,\Delta_{\lambda}\otimes_{\cA_\la}C).\]
  Moreover we have the homomorphism in the category $\fC^\star$: 
  \begin{equation}\label{M}\Delta_{\lambda}\otimes_{\cA_\la} r_{\lambda}(M)\rightarrow M.\end{equation}
In fact, $r_{\lambda}(M)$ is a $\cA_\la$ submodule in the weight $\la$ subspace of $M$.
Now $\Delta_{\lambda}\otimes_{\cA_\la}r_{\lambda}(M)$ also contains a natural subspace isomorphic to $r_{\lambda}(M)$ and 
$\Delta_{\lambda}\otimes_{\cA_\la}r_{\lambda}(M)$ is generated from this subspace as a $\U(\fg)$ module (since $\Delta_{\lambda}$ is cyclic). Now the definition of $\Delta_{\lambda}$ implies that 
as a $\fg$ module $\Delta_{\lambda}\otimes_{\cA_\la}r_{\lambda}(M)$ is isomorphic to 
$i_\la^*(\U(\fg)\otimes_{\U(\overline \fg)} r_{\lambda}(M))$. However, $M\in\fC^\star_{\leq\la}$, hence we obtain the desired homomorphism \eqref{M}.

Thus we have the natural map
  \[\eta:\Hom_{\cA_\la}(B,r_{\lambda}(M))\rightarrow \Hom_{\fC^\star_{\leq\lambda}}(\Delta_{\lambda}\otimes_{\cA_\la}B, M).\]
 We need to show that $\eta\circ\tau=\Id_{\fC^\star_{\leq\la}}$ and $\tau\circ\eta=\Id_{\cA_\la\rm{-mod}}$.  
First, we note that the composition  $\tau\circ\eta$ is equal to the identity map by construction. Second, let us consider the opposite composition $\eta\circ\tau$.  The map 
$\tau:\Hom_{\fC^\star_{\leq\lambda}}(\Delta_{\lambda}\otimes_{\cA_\la}B, M)\rightarrow \Hom_{\cA_\la}(B,r_{\lambda}(M))$
send a homomorphism $\varphi\in \Hom_{\fC^\star_{\leq\lambda}}(\Delta_{\lambda}\otimes_{\cA_\la}B, M)$ to its restriction to the weight $\lambda$ parts. After applying $\eta$ we get our $\phi$ back, since a map from $\Delta_{\lambda}\otimes_{\cA_\la}B$ is determined by its restriction to the highest weight part. This proves the equality $\eta\circ\tau=\Id_{\fC^\star_{\leq\la}}$ and completes the proof of our Lemma.
\end{proof}

\begin{prop}
\label{prp::Ext::propers}
If categories $\fC$ and $\fC^\star$ are stratified then
for all $\lambda\neq\mu$ we have the vanishing of all higher extension groups 
$$\Ext_{\fC}(\bar{\Delta}_{\la},\bar{\nabla}_{\mu}) = 
 \Ext_{\fC^\star}(\bar{\Delta}_{\la},\bar{\nabla}_{\mu}) = H_{Lie}^{\udot}(\fg,\fg_{0};\Hom(\bar{\Delta}_\lambda;\bar{\nabla}_{\mu}))=0.
 $$
\end{prop}
\begin{proof}
Note that thanks to the vanishing of Ext's between proper standards and costandard in $\fC$ and respectively vanishing of Ext's between standard and proper costandard in $\fC^\star$ it 
is sufficient to show that 
$\bar{\Delta}_\la$ admits a projective resolution in $\fC_{\leq\la}$ whose indecomposable projective summands are only $\Delta_\la$ and 
the module $\bar{\nabla}_{\la}$ admits an injective resolution whose only indecomposable summands are isomorphic to $\nabla_{\la}$.
We prove the existence of these resolutions by induction with respect to the partial ordering on weights.

The base of induction is obvious, because when $\lambda$ is minimal then the only indecomposable projective in $\fC_{\leq\la}$ is isomorphic to $\Delta_{\la}$. 
For the induction step, we suppose that for all $\mu<\lambda$ the proper standard module $\bar\nabla_{\mu}$ admits an resolution by costandards $\nabla_{\mu}$. 
Consequently, $\Ext_{\fC}(\bar{\Delta}_{\la},\bar{\nabla}_{\mu})$ vanishes. 
But both proper standards and proper costandard modules are finite-dimensional and, in particular, belong to $\fC$ and $\fC^\star$ and the Ext's groups coincides with appropriate Lie algebra cohomology. What follows vanishing of Ext's in the dual category $\fC^\star$.  
Now consider the minimal projective resolution of 
of $\bar\Delta_{\la}$ in $\fC_{\leq\la}^{\vee}$, the projective summands here should be either $\Delta_{\la}$ or $\imath_{\lambda}^*(P_{\mu})$ with $\mu<\lambda$. However, the existence of the latter summand $\imath_{\lambda}^*(P_{\mu})$ will imply the nonzero Ext to $\bar\nabla_{\mu}$ which is supposed to be zero.
\end{proof}

\begin{cor}
\label{cor::r::exact}
If $\fC$ and $\fC^\star$ are stratified, then functors $r_{\la}^{*}$ and $r_{\la}^{!}$ are exact.
\end{cor}
\begin{proof}
The functor $r_{\la}^{*}$ is left adjoint to the exact functor $r_{\la}$ and admits a left derived functor $L_{\ldot}r_{\la}^{*}$.
The irreducibles in the category  $\cA_\lambda$-mod we are working with (graded $\cA_\lambda$-modules bounded from below or from above)
are isomorphic to $\Bbbk$ shifted in appropriate degree. Therefore, it is enough to show that higher derived $L_{>0}r_{\la}^{*}(\Bbbk)$ vanishes on the trivial module $\Bbbk$.

On the other hand, as we have mentioned above, Proposition~\ref{prp::Ext::propers} is 
equivalent to the existence of a projective resolution of the proper standard $\bar\Delta_{\la}$ by standards $\Delta_{\la}$. 
Let us denote this resolution by $Q_{\ldot} \rightarrow \bar\Delta_{\la}$. The functor $r_{\lambda}$ is exact and $r_{\lambda}(\Delta_{\lambda})=\cA_{\lambda}$. Therefore, $r_{\lambda}(Q_{\ldot})$ is a free resolution of $\Bbbk$ in the category of $\cA_{\lambda}$-mod. 
On the other hand, $r_{\lambda}^{*}(r_{\lambda} (Q_{\ldot})) = Q_{\ldot}$. Therefore, the left derived functors $L_{>0}{r_{\lambda}^{*}}(\Bbbk)$ vanish.
\end{proof}
The following corollary is an application of the classical homological algebra and satisfied under assumptions of Corollary~\ref{cor::r::exact} -- $\fC$ and $\fC^\star$ are stratified.
\begin{cor}
\label{cor::global::local}
If the standard module
 $\Delta_{\lambda}$ is finitely presented $\cA_{\lambda}$ module, then it is projective $\cA_{\lambda}$-module.
Moreover, if $\cA_\la$ is a polynomial ring, then we have an isomorphism of $\cA_{\la}$-modules:
$$
 \Delta_{\lambda} \simeq \cA_{\lambda}\otimes \bar\Delta_{\lambda} \text{ as $\cA_{\la}$-mod.}
$$
Similarly,
finitely presented $\cA_{\la}$-module $\nabla_\lambda^{\vee}$ is projective and for polynomial $\cA_{\lambda}$ this module is free:
$$
 \nabla_{\lambda}^{\vee} \simeq \cA_{\lambda} \otimes \bar{\nabla}_{\lambda}^{\vee}
 \text{ as $\cA_{\la}$-mod.}
$$
\end{cor}
\begin{proof}
The standard module $\Delta_\lambda$ is flat $\cA_{\lambda}$-module, because the functor $r_{\lambda}^{*}(\texttt{-})=\Delta_{\lambda}\otimes_{\cA_{\lambda}}\texttt{-}$ is exact by Corollary~\ref{cor::r::exact}.
A flat finitely presented module is projective (see e.g.~\cite{Roitman}) and projective module over polynomial ring is free thanks to the famous Quillen-Suslin theorem (\cite{Suslin-Quillen}).
\end{proof}

\subsection{Filtrations and bimodules}
\label{FandBi}
Within subsection~\S\ref{FandBi} we assume that
\begin{equation}
\label{SCat}
\fC\text{ and } \fC^\star \text{ are stratified with respect to partial orderings $\preceq$ and $\preceq^\vee$.}
\end{equation}

Consider the left $\fg$-module structure on $\Bbbk[G]$ and  the filtration
\[
\mathcal{F}_\lambda=\imath_{\leq\la}^! \Bbbk[G].
\]
In other words, 
$\mathcal{F}_\lambda$ is the maximal  left weight submodule of  $\Bbbk[G]$ 
with the property that all the nontrivial dominant weights $\mu$ of  $\mathcal{F}_\lambda$ satisfy $\mu\preceq\la$.
Since the left weights of $u$ and $ux$ are equal for any $u \in\Bbbk[G]$, $x \in \fg$, the space  $\mathcal{F}_\lambda\subset \Bbbk[G]$ is a sub-bimodule. By definition we have
\begin{equation}\label{InclusionFiltrations}
   \mathcal{F}_\lambda\supset \mathcal{F}_\mu,~\text{for } \lambda \succ \mu.
\end{equation}

\begin{lem}\label{Flambda}
One has the isomorphism of the left modules \[  \mathcal{F}_\lambda=\bigoplus_{\mu \in P_+(\fg_0)}\imath_{\leq\la}^!\mathbb{I}_\mu\otimes (L_{-w_0\mu})^o = \bigoplus_{\mu \preceq \lambda}\imath_{\leq\la}^!\mathbb{I}_\mu\otimes (L_{-w_0\mu})^o. \].
\end{lem}
\begin{proof}
We apply functor $\imath_\la^!$ to Proposition \ref{moduleFunctions Decomposition} and obtain the first equality. The second equality follows from the fact that  $\imath_\la^!\mathbb{I}_\mu=0$ unless 
$\mu \preceq \lambda$.    
\end{proof}

Recall the Peter-Weyl decomposition
\[
\Bbbk[G_0]\simeq \bigoplus_{\la\in P_+(\fg_0)} L_\la\T (L_{-w_0\la})^o. 
\]
Let $v_\la\T v_\la^o$ be the tensor product of  weight $\lambda$ vectors in each summand.  Now let us consider the subspace $\mathcal{F}^{hw}_\lambda\subset \mathcal{F}_\lambda$ consisting of vectors of bi-weight $(\la,\la)$ (recall that $\mathcal{F}_\lambda\subset \Bbbk[G]$ and hence we have the natural right and left actions of $\U(\fg)$).
In particular, $v_\la\T v_\la^o\in \mathcal{F}^{hw}_\lambda$.

Recall that  $\overline{\fg}\subset \fg$ is the $\fh$-weight zero subalgebra of $\fg$.
By definition, $\mathcal{F}^{hw}_\lambda$ admits a natural structure of $\U(\overline{\fg})$ bimodule. By properties of the order $\prec$ we have that:
\begin{equation}\label{HighestWeightAlgebraInclusion}
\mathcal{F}^{hw}_\lambda \subset\U(\overline{\fg})^\vee \otimes \Bbbk (v_\la\T v_\la^o).  
\end{equation}
Indeed, thanks to condition \eqref{assumption} for any $r\in \mathcal{F}^{hw}_\lambda$ there exists an element $u\in \U(\overline{\fg})$ such that
$u.r=v_\la\T v_\la^o$. In fact, since $\Bbbk[G]$ is coinduced from its subspace $\Bbbk[G_0]$, there exists $u'\in\U(\fg)$ such that $u'.r=v_\la\T v_\la^o$;
in particular, weight of $u'$ is equal to zero.
Since all the weights of $\mathcal{F}_\lambda$ are less than or equal to $\la$, the condition \eqref{assumption} implies that the existence of $u'$ as above leads to the existence of $u\in \U(\overline{\fg})$ such that
$u.r=v_\la\T v_\la^o$.

Recall \eqref{abelian} that we assume that $\overline{\fg}$ is abelian. 

\begin{lem}
For any $\la\in P_+(\fg_0)$ the left and right actions of $\U(\overline{\fg})$ on $\mathcal{F}^{hw}_\lambda$ coincide.
\end{lem}
\begin{proof}
Let $\overline{G}\subset G$ be the abelian unipotent Lie group of the (possibly infinite dimensional) Lie algebra $\overline{\fg}$.
Also let $(v_\la\T v_\la^o)^*$ be a linear form on  $\mathcal{F}^{hw}_\lambda$ which returns the zero-degree coefficient of a function (in other words, the coefficient in front of $v_\la\T v_\la^o$).  

We prove lemma in two steps: first, we prove the coincidence of the matrix coefficients for the left and right actions of the group  $(v_\la\T v_\la^o)^*(\bar g.F)=(v_\la\T v_\la^o)^*(F.\bar g)$ for any $\bar g\in \overline G$ and $F\in \mathcal{F}^{hw}_\lambda$. 
After that we deduce the desired statement.

For the first step 
we need to show that for any $\bar g\in\overline{G}$ and $F\in \mathcal{F}^{hw}_\lambda$ one has $(v_\la\T v_\la^o)^*(\bar g.F.\bar g^{-1})=(v_\la\T v_\la^o)^*F$. Since $\overline G \subset R$,
the conjugation by $\bar g$ stabilizes $\Bbbk[G/R]$. This completes the first step. 

It is immediate from the first step that for any $a\in \overline\fg$ and $w\in \mathcal{F}^{hw}_\lambda$ one has
\begin{equation}\label{lr}
(v_\la\T v_\la^o)^*(a.w)=(v_\la\T v_\la^o)^*(w.a).
\end{equation}

Now let us consider $\mathcal{F}^{hw}_\lambda$ as a bimodule over $\fg$. For an element $a\in\overline \fg$ let $r(a)-l(a)$ be the difference between the right and left actions as an operator on $\mathcal{F}^{hw}_\la$. Formula \eqref{lr} implies that 
$v_\la\T v_\la^o\notin {\rm Im} (r(a)-l(a))$.
Now take an element $F\in \mathcal{F}^{hw}_\lambda$. Then from the commutativity of $\overline\fg$ one gets that 
\[
v_\la\T v_\la^o\notin
\U(\overline\fg)\bigl((r(a)-l(a))F\bigr).
\]
However, $\mathcal{F}^{hw}_\lambda$ is cocyclic and hence  $(r(a)-l(a))F=0$. This implies the desired statement. 
\end{proof}

We consider the $\fg\de\fg$ bimodule
\[
{\mathbb T}_\lambda:=\nabla_\lambda \otimes_{\mathcal{A}_\lambda}(\Delta_\lambda^\vee)^o.
\]

\begin{lem}
Assume that $\cA_\la$  is a polynomial algebra for all  $\la$. Then as a $\fg\de\fg$ bimodule ${\mathbb T}_\lambda$ is defined by the following relations:
\begin{itemize}
\item
${\mathbb T}_\lambda$ is a cocyclic bimodule with cocyclic vector $w_\la$ of bi-weight $(\la,\la)$,
\item all the bi-weights $(\mu,\nu)$ of ${\mathbb T}_\lambda$ satisfy $\mu\preceq\la$, $\nu\preceq\la$, 
\item the right and left actions of $\U(\overline\fg)$ on the $(\la,\la)$-weight subspace coincide.
\end{itemize}
\end{lem}
\begin{proof}
Let $\overline{\mathbb T}_\lambda$ be the $\fg\de\fg$ bimodule defined by the above relations. One easily sees that there is an embedding of bi-modules 
${\mathbb T}_\lambda\to \overline{\mathbb T}_\lambda$. Now we show that the "size" of $\overline{\mathbb T}_\lambda$ is the expected one.

First note that the left-weight $\la$ subspace of $\overline{\mathbb T}_\lambda$ embeds into $(\Delta_\lambda^\vee)^o$ as a right module. Let $\{b_i\}_{i=1}^N$, $N=\dim (\overline \Delta_\lambda^\vee)^o$ be a set of cogenerators of $(\Delta_\lambda^\vee)^o$ as an $\mathcal{A}_\lambda$-module.  In particular, by Corollary \ref{cor::global::local},
the character $\mathrm{span}\{b_i\}_{i=1}^N$ coincides with the character of $(\overline{\Delta}_\lambda^\vee)^o$.
We use the same notation $b_i$ to denote the corresponding (left-weight $\la$) vectors in $\overline{\mathbb T}_\lambda$.
Then $\overline{\mathbb T}_\lambda$ is cogenerated by $\{b_i\}_{i=1}^N$ as a left $\fg$-module. However, the left submodule of $\overline{\mathbb T}_\lambda$ cogenerated by each vector $b_i$ is a quotient of $\nabla_\la$ (because of the third defining condition). Hence the character of $\overline{\mathbb T}_\lambda$ is at most the character of ${\mathbb T}_\lambda$.
\end{proof}

\begin{thm}\label{FT}
Assume that $\cA_\la$  is a polynomial algebra for all  $\la$. Then we have the following isomorphism of bimodules:
    \[\mathcal{F}_\lambda \left/ \sum_{\mu \prec \lambda}\mathcal{F}_\mu \right. \simeq {\mathbb T}_\la.\]
In other words, we have an isomorphism for the associated graded to this filtration:
$$
{\rm gr}\ \! \Bbbk[G] \simeq \bigoplus_{\la\in P_+} \nabla_\la\T_{\cA_\la} (\Delta_\la^\vee)^o.
$$
\end{thm}
\begin{proof}
Due to Lemma \ref{Flambda} the right weights of elements of $\mathcal{F}_\lambda$ are $\preceq \lambda$. Now Lemma \ref{Flambda} and BGG reciprocity imply the following isomorphism of the left modules:
\begin{equation}\label{Ftonabla}
\mathcal{F}_\lambda \left/ \sum_{\mu \prec \lambda}\mathcal{F}_\mu \right. \simeq \nabla_\lambda \otimes_{\Bbbk}\overline \Delta_\lambda^\vee.
\end{equation}
In fact, 
\[
 \mathcal{F}_\lambda \simeq 
 \bigoplus_{\nu \preceq \lambda}\imath_{\leq\la}^!\mathbb{I}_\nu\otimes (L_{-w_0\nu})^o.
\]
Now
\[
\mathcal{F}_\lambda \left/ \sum_{\mu \prec \lambda}\mathcal{F}_\mu \right. \simeq \bigoplus_{\nu\preceq\lambda} \left(\imath_{\leq\la}^!\mathbb{I}_\nu\left/\sum_{\nu\preceq\mu\prec\la} \imath_{\leq\mu}^!\mathbb{I}_\nu\right. \right)\otimes (L_{-w_0\nu})^o.
\]
By BGG reciprocity (Corollary \ref{BGGreciprocity})
\[\imath_{\leq\la}^!\mathbb{I}_\nu\left/\sum_{\nu\preceq\mu\prec\la} \imath_{\leq\mu}^!\mathbb{I}_\nu\right. \simeq [\overline \Delta_\lambda^\vee : L_\nu]\nabla_\lambda.\]
Now summing up over all $\nu$ we obtain \eqref{Ftonabla}.

Now we note that     
$\mathcal{F}_\lambda \left/ \sum_{\mu \prec \lambda}\mathcal{F}_\mu \right.$ is a cocylic $\fg\de\fg$ bimodule. Moreover, one easily checks that all the conditions from the definition of the bimodule $\overline{\mathbb T}_\la$ are satisfied for  $\mathcal{F}_\lambda \left/ \sum_{\mu \prec \lambda}\mathcal{F}_\mu \right.$. Hence we obtain an embedding of bimodules
\[
\mathcal{F}_\lambda \left/ \sum_{\mu \prec \lambda}\mathcal{F}_\mu \right.\subset \overline{\mathbb T}_\la,\quad v_\la\T v_\la^o\mapsto w_\la.
\]
Since the characters of these bimodules are the same, they must coincide.
\end{proof}

\begin{cor}
    Assume that $\fg$ is finite dimensional Lie algebra with highest weight category of representations with respect to some order $\prec$ on the set of weights $\Theta=P_+(\fg_0)$ and $\overline \fg=\fh$. Then the following holds for the bimodule of functions on the corresponding connected group:
      \[\mathcal{F}_\lambda \left/ \sum_{\mu \prec \lambda}\mathcal{F}_\mu \right. \simeq {\mathbb T}_\la\simeq 
       \nabla_\lambda \otimes_{\Bbbk}\Delta_\lambda^\vee.\]
\end{cor}
\begin{proof}
    All the assumptions in the beginning of Section \ref{progroup} hold for finite dimensional Lie algebras. The property $\overline \fg=\fh$ implies  $\mathcal{A}_\lambda \simeq \Bbbk$ and $\Delta_\lambda=\overline\Delta_\lambda$, $\nabla_\lambda=\overline\nabla_\lambda$.
\end{proof}

\begin{rem}\label{VanDerKallen}
In the case when $\mathfrak{g}=\mathfrak{b}$ and $\mathfrak{g}_0=\mathfrak{h}$ are Borel and Cartan subalgebras in a reductive Lie algebra, Corollary 2.21 recovers a result of van der Kallen \cite{vdK}. The relevant order $\prec$ is that of Definition~\ref{ChOrder} below, and the modules $\Delta_\lambda=\overline{\Delta}_\lambda$ and $\nabla_\lambda=\overline{\nabla}_\lambda$ are the usual (finite-dimensional) Demazure and van der Kallen modules, respectively. To see the latter, one argues as in the proof of Theorem~\ref{DUStCost} below, using presentations for the modules $\Delta_\lambda$ and $\nabla_\lambda$ as cyclic $\mathfrak{b}$-modules (which are obtained by omitting relations of nonzero $z$-degree from those in Section~\ref{Relations}). The characters of the $\nabla_\lambda$ are known as Demazure atoms.
\end{rem}

\section{Nonsymmetric Macdonald polynomials}\label{MacdonaldPolynomials}

In this section we collect basic notation and constructions on the nonsymmetric Macdonald polynomials (see \cite{M1,M2,Ch2,CO1,CO2,OS}). 

\subsection{Affine Weyl groups}
Let $\fs=\fn^+\oplus \fh\oplus \fn^-$ be a finite dimensional simple Lie algebra with root system $\Delta=\Delta_+\cup\Delta_-$, dual root system $\Delta^\vee$, Weyl group $W$, root lattice $Q$, coroot lattice $Q^\vee$, and weight lattice $P$. 
We consider the affine  and extended affine Weyl groups
\[W^{a\vee}:=W \ltimes Q,\qquad W^{e\vee}:=W \ltimes P.\]
For $\mu \in P$, we denote by $t_{\mu}$ the corresponding translation element of $W^{e\vee}$. 
More precisely, $W^{e\vee}$ (resp., $W^{a\vee}$) consists of elements $t_{\mu}\sigma$ for $\sigma \in W$, $\mu \in P$ (resp., $\mu \in Q$). For $\sigma, \tau \in W$,
$\mu, \nu \in P$, we have
\[t_{\mu}\sigma t_{\nu}\tau=t_{\mu +\sigma(\nu)}\sigma\tau.\]
We use the notation 
\[\wt(t_{\mu}\sigma)=\mu,\qquad \dir(t_{\mu}\sigma)=\sigma.\]
Note that $\dir$ is a surjective homomorphism $W^{e\vee}\twoheadrightarrow W$.

By definition, we have that $W^{a\vee} \subset W^{e\vee}$ is a normal subgroup and
\[W^{e\vee} / W^{a\vee} \simeq P/Q=:\Pi.\]
Moreover, we can realize $\Pi$ as a subgroup of $W^{e\vee}$ such that
$W^{e\vee} \simeq \Pi \ltimes W^{a\vee}.$


Let $s_1^\vee, \dots, s_n^\vee$ be the simple reflections in $W$ and set $I=\{1,\dotsc,n\}$. Let $\theta$ be the highest short root of $\Delta$; then $\theta^\vee$ is the highest (long) coroot of $\Delta^\vee$. We denote
\[s_0^\vee=t_{\theta}s_{\theta^\vee}\in W^{a\vee},\qquad \alpha_0^\vee=-\theta^{\vee}+\delta^\vee\in Q^{a\vee},\]
where $Q^{a\vee}=Q^\vee\oplus\mathbb{Z}\delta^\vee$.

Then each element of $W^{a\vee}$ can be expressed as a product of reflections $s_i, i=0, \dots, n$ and each element   $\tau \in W^{e\vee}$ has the expression of the form $\tau=\pi s_{i_1}, \dots, s_{i_l}$, $\pi \in \Pi$. If such $l$ is the smallest possible number of simple reflections in the decomposition of $\tau$ then it is called the length of $\tau$ and is denoted by $l(\tau)$.

For an element $\beta^\vee\in Q^{a\vee}$, we introduce the notation $\beta^\vee = \bar\beta^\vee+ \deg(\beta)\delta^\vee$ where $\bar\beta^\vee\in Q^\vee$; we call $\bar\beta^\vee$ the classical part of $\beta^\vee$. For any affine coroot $\beta^\vee=\overline{\beta}^\vee+k\delta^\vee\in \Delta^{a\vee}:=\Delta^\vee\times \mathbb{Z}\delta^\vee$, we have an affine reflection $s_{\beta^\vee}=t_{-k\overline{\beta}}s_{\overline\beta^\vee}$.

For an element $\la\in P$, let $m_\lambda=t_\lambda \sigma_{\lambda}$ be the unique shortest element in $t_\lambda W$. Then $\sigma_\lambda\in W$ is characterized as the unique shortest element such that 
$\lambda=\sigma_{\lambda}(\lambda_-)$, where $\lambda_-\in P_-$ is the unique antidominant weight in the $W$-orbit of $\lambda$.

For any $i=1,\dotsc,n$ and $\sigma\in W$, we denote
\[\delta_{i,\sigma}=\begin{cases}
    0, ~\text{ if } \sigma(\alpha_i) \in \Delta_+,\\ 
    1, ~\text{ if } \sigma(\alpha_i) \in \Delta_-.
\end{cases}\]

\subsection{Nonsymmetric Macdonald polynomials}\label{NMp}



Let $\ge$ denote the Bruhat order on $W$.

\begin{dfn}\label{ChOrder}
The Cherednik partial order $\succeq$ on $P$ is defined as follows:  $\lambda \succeq \mu$ if and only if $\la_--\mu_-\in -Q_+$ and $\sigma_\lambda \leq \sigma_\mu$ when $\la_-=\mu_-$.

The dual Cherednik partial order $\succeq^\vee$ is defined as follows: $\lambda \succeq^\vee \mu$ if and only if $\la_--\mu_-\in -Q_+$ and $\sigma_\lambda \geq \sigma_\mu$ when $\la_-=\mu_-$.
\end{dfn}


Let $\QQ(q,t)[P]$ be the group algebra of the weight lattice over the field $\QQ(q,t)$ of rational functions in $q,t$. For each fundamental weight $\omega_i$, $i=1,\dots, n$
we denote by $x_i$ the formal exponential of $\omega_i$. Then
\[\QQ(q,t)[P]=\QQ(q,t)[x_i^{\pm1}]_{i=1, \dots,n}.\]
For $\lambda=\sum_{i=1}^n \lambda _i\omega_i \in P$, we write
\[x^{\lambda}=\prod_{i=1}^nx_i^{\lambda_i}.\]

Let
\[\mu=\prod_{\alpha \in \Delta_+}\prod_{i=0}^\infty\frac{\left(1-q^i x^{\alpha}\right)\left(1-q^{i+1} x^{-\alpha} \right)}{\left(1-tq^i x^{\alpha}\right)\left(1-tq^{i+1} x^{-\alpha} \right)}\in \ZZ[P][[q,t]] \subset \QQ[[q,t,P]],
\]
where we denote by $\QQ[[q,t,P]]$ the vector space consisting of all formal series $\sum_{a,b\in\ZZ,\lambda\in P} c_{a,b,\lambda}q^a t^b x^\lambda$.
For $f\in \QQ[[q,t,P]]$ we denote by $\langle f \rangle \in \QQ[[q,t]]$ the constant term of $f$, i.e., the coefficient of $x^0$ in $f$. It is known (see, e.g., \cite[Theorem 6.3]{H} for a proof) that
\[\mu_\circ:=\mu/\langle \mu\rangle \in \QQ(q,t)[[P]].\]
We define the scalar product on $\QQ(q,t)[P]$:
\[\langle f,g \rangle=\langle fg\mu_\circ \rangle.\]


The ring $\Bbbk(q,t)[P]$ has the monomial basis $\{x^{\lambda}\}_{\lambda \in P}$ over $\Bbbk(q,t)$. By \cite[\S 4, (4.4)]{Ch1}, there exists a unique basis $\{E_{\lambda}(x,q,t)\}_{\lambda\in P}$ of $\Bbbk(q,t)[P]$ which is lower unitriangular with respect to monomials under the partial order $\prec$ and satisfies the orthogonality relations
\[\langle  E_{\lambda}(x,q,t),E_{\mu}(x^{-1},q^{-1},t^{-1}) \rangle\neq 0 \quad \Leftrightarrow \quad \la=\mu. \]
The notation $f(x^{-1})$ means that we replace all $x^\nu$ by $x^{-\nu}$ ($\nu\in P$) in $f\in\Bbbk(q,t)[P]$.



The specialized Macdonald polynomials $E_{\lambda}(x,q,0)$ and $E_{\lambda}(x^{-1},q^{-1},\infty)$ are known to be well-defined and belong to $\mathbb{Z}[q][P]$ by \cite[Theorem 1]{I} or \cite[Corollary 3.6]{CO1} (and, in fact, they belong to $\mathbb{Z}_{\ge 0}[q][P]$ by \cite[Theorem 1]{I} or \cite[Corollary 4.4]{OS}). By specializing the orthogonality relations \cite[Main Theorem]{Ch1}, we have
\begin{equation}\label{EOrth0}
\langle E_{\lambda}(x,q,0),E_{\mu}(x^{-1},q^{-1},\infty) \rangle_0=(q)_\lambda\delta_{\lambda,\mu},
\end{equation}
where, for $f,g\in\Bbbk(q)[P]$,
\[\langle f, g \rangle_0 := \langle f,g\rangle|_{t=0} = (q;q)_\infty^n\langle f g \prod_{\alpha \in \Delta_+}\prod_{i=0}^\infty\left(1-q^i x^{\alpha}\right)\left(1-q^{i+1} x^{-\alpha} \right)\rangle
\]
and
\begin{equation}\label{Norm}
(q)_\lambda = \prod_{i=0}^n\prod_{j=1}^{\lambda_i-\delta_{i,\sigma}}\left(1-q^j\right)
\end{equation}
for $\lambda=\sum_{i=1}^n \lambda _i\omega_i \in P$.

\begin{rem}
The factor $(q;q)_\infty^n$ is $\langle \mu\rangle|_{t=0}^{-1}$.
\end{rem}

\subsection{Quantum Bruhat graphs}\label{qBg}

The Bruhat graph BG of $W$ (see e.g. \cite{BB}) is the directed graph whose set of vertices is identified with $W$ and we have an
arrow $w \to w s_\al$
for $w\in W$ and $\al \in \Delta_+$ if and only if $\ell ( w s_\al )=\ell ( w )+1$.
The quantum Bruhat graph QBG of $W$ (see e.g. \cite{BFP,LNSSS1}) is an enhancement of BG obtained by adding
a ``quantum" arrow $w \to w s_\al$ for each $w\in W$ and $\al \in \Delta_+$ so that
$$\ell ( ws_\al ) = \ell ( w ) - \sum_{\gamma \in \Delta_+} \bra \gamma, \al^{\vee} \ket + 1.$$

Assume that we are given an element $z_0\in W^e$ and a sequence of affine coroots
$\beta_1^\vee,\dots,\beta_l^\vee$. A path $p_J$ ($=:p$) corresponding to a set
$$J=\{1\le j_1<\dots <j_r\le l\} \subset 2^{[1,l]}$$
is a sequence $p_J=(z_0,z_1,\dots,z_r)$, where $z_{k+1}=z_ks_{\beta_{j_{k+1}}}$. Here we refer the last element $z_r$ as the end of the path $p$, and denote it by ${\rm end}(p)$ ($= {\rm end}(p_J) =z_r$). We denote
\[\wt(p):=\wt({\rm end}(p)).\]
We say that $p$ is a quantum alcove path if 
\[
{\rm dir}(z_0) 
\longrightarrow
{\rm dir}(z_1) 
\longrightarrow
\dotsm 
\longrightarrow
{\rm dir}(z_r)
\]
is a path in QBG.
For an alcove path $p$ we denote by $\qwt(p)$ the sum of all $\beta_{j_k}^\vee$ such that the edge
${\rm dir}(z_{k-1}) 
\longrightarrow {\rm dir}(z_k)$ is quantum.

For any $u\in W^e$, one denotes by ${\mathcal{QB}}(\id; u)$ the set of quantum alcove paths
with $z_0=\id$ and with affine coroots coming from a fixed reduced decomposition of $u=\pi s_{i_1}\dots s_{i_l}$:
\begin{equation}\label{beta-u}
\beta_k^\vee(u) = s_{i_l}\dots s_{i_{k+1}} \al_{i_k}^\vee.
\end{equation}
Note that the coroots $\beta^\vee_k(u)$ as well as the set ${\mathcal{QB}}(\id; u)$ depend
on a reduced expression of $u$, but we omit this from the notation. We also write simply $\beta^\vee_k$ when the element $u$ is clear from context.

By means of the Ram-Yip formula for $E_\la(x,q,t)$ \cite{RY}, the following formula for $t=0$ specialized nonsymmetric Macdonald polynomials was proved in \cite[Corollary 4.4]{OS}:
\begin{equation}\label{E0}
    E_{\lambda}(x,q,0)=\sum_{p \in {\mathcal{QB}}(\id; m_{\lambda})}x^{\wt(p)}q^{\deg(\qwt(p))},
\end{equation}
where we recall that $m_\lambda$ is the minimal representative of the coset $t_\lambda W$.
A similar formula for $E_\lambda(x^{-1},q^{-1},\infty)$ was also proved in \cite[Proposition 5.4]{OS}; we will not need this formula explicitly in this paper.

Recall the decomposition $t_\lambda = m_\lambda \sigma_\lambda^{-1}$ for $\lambda\in P$. By conjugation, we can also write this as $t_{\lambda_-}=\sigma_\la^{-1} m_\la$.
It is well known that this decomposition is reduced, i.e., $\ell(t_{\la_-})=\ell(\sigma_\la^{-1})+\ell(m_\la)$.
Then if we take reduced expressions
\[
\sigma_\la^{-1}=s_{i_1}\dots s_{i_r},\ m_\la=\pi s_{i_{r+1}}\dots s_{i_M},
\]
then we obtain a reduced expression
\begin{equation}\label{reddecV}
t_{\la_-}=\pi s_{\pi^{-1}i_1}\dots s_{\pi^{-1} i_r} s_{i_{r+1}}\dots s_{i_M}.
\end{equation}
Hence, with these choices of reduced expressions, we have
\begin{equation}\label{tm}
\beta^\vee_j(m_\la)=\beta^\vee_{j+r}(t_{\la_-}),\qquad j=1,\dots,\ell(t_{\la_-})-r.
\end{equation}

\section{Representations of the Iwahori algebra}\label{Iwahori}

\subsection{Iwahori algebra} \label{curalg}
Let $\mathfrak{s}[z]=\mathfrak{s}\T \bC[z]$ be the current algebra of a simple Lie algebra $\mathfrak{s}$. We have a grading on $\mathfrak{s} [z]$ by setting $\deg \, a \otimes z^m = m$ for each $a \in \mathfrak{s} \setminus \{ 0 \}$ and $m \ge 0$. We use the notation $a z^m =a \otimes z^m$, $a=a \otimes 1$. 

Let $e_\alpha, f_{-\alpha}, h_\alpha=-h_{-\alpha}$ for $\alpha\in \Delta_+$ be a set of Cartan generators for $\mathfrak{s}$. We continue to use the notation for objects associated with $\fs$ from the previous section.
We denote by $\mathcal{I} = \fn^+\oplus\fh \oplus(\mathfrak{s}\T z\bC[z])$ the Iwahori algebra, which is a subalgebra of the current algebra $\mathfrak{s}\T\bC[z]$. Let $S$ be the connected simply connected Lie group of $\mathfrak{s}$. The Iwahori group 
$\bfI\subset S[[z]]$ is defined as the preimage of the Borel subgroup 
$B\subset S$ under the $z=0$ evaluation map $S[[z]]\to S$. We note that the Lie algebra of $\bfI$ is a completion of $\mathcal{I}$ (see Remark \ref{rem:Iwahori}).

\begin{rem}
In this paper we only consider the case of simply connected $S$. In general, we expect that analogues of our results hold true for any
Lie group of the Lie algebra $\mathfrak{s}$. However, in the general case 
the structure of the corresponding global Weyl modules is more complicated.
\end{rem}

An $\mathcal{I}$-module $M$ is called graded if $M=\bigoplus_{j\in \bZ} M_j$ such that each $M_j$ is $\fh$ semi-simple with finite-dimensional weight spaces and $(a\T z^i) M_j\subset M_{i+j}$ for all $a\otimes z^i\in\mathcal{I}$. We define the character of $M$ as the formal linear combination
\[
\ch \, M=\sum_{j\in\bZ} \ch \, M_j\, q^j,
\]
where $\ch \, M_j$ is the $\fh$-module character. In what follows we always consider the modules $M$ whose
$\fh$-weights belong to $P$. Correpondingly, we have $\ch \, M \in \bZ[[P]][[q]]$. 
For such $M$, we denote by $M^\vee$ the restricted dual space $M^\vee=\bigoplus_{j\in \ZZ} (M^\vee)_j$ where $(M^\vee)_j=\bigoplus_{\la\in P}  ({}_{-\lambda} M_{-j})^*$. 

\begin{rem}
We use the notation ${}_\lambda M$ for the weight space (with respect to the left action) to reserve the right  lower index for the right weight whenever our space $M$ is endowed with a bimodule structure.   
\end{rem}

Recall the $W$-action on the set
$\{e_\al\}_{\al\in\Delta_+}\cup\{f_{-\al}\T z\}_{\al\in\Delta_+}$ following \cite{FeMa}: for an element
$\sigma\in W$
and $\al\in\Delta_+$ we set
\[
\widehat{\sigma} e_\al=\begin{cases}
e_{\sigma(\al)}, \qquad  \sigma(\al)\in\Delta_+,\\
f_{\sigma(\al)}\T z, \ \sigma(\al)\in\Delta_-,
\end{cases}
\widehat{\sigma} (f_{-\al}\T z)=\begin{cases}
e_{-\sigma(\al)}, \qquad  \sigma(\al)\in\Delta_-,\\
f_{-\sigma(\al)}\T z, \ \sigma(\al)\in\Delta_+.
\end{cases}
\]
We also use the following notation for $\al\in\Delta_+$ and $r\ge 0$:
\begin{align*}
e_{\widehat{\sigma}(\al)+r\delta}&=\begin{cases}
e_{\sigma(\al)}\T z^r, &  \sigma(\al)\in\Delta_+,\\
f_{\sigma(\al)}\T z^{r+1}, & \sigma(\al)\in\Delta_-,
\end{cases},\\
e_{\widehat{\sigma}(-\al+\delta)+r\delta}&=\begin{cases}
e_{-\sigma(\al)}\T z^r, &  \sigma(\al)\in\Delta_-,\\
f_{-\sigma(\al)}\T z^{r+1}, & \sigma(\al)\in\Delta_+.
\end{cases}
\end{align*}

\begin{dfn}
Suppose $\lambda\in P$ and $\sigma\in W$ satisfy $\lambda=\sigma(\lambda_-)$. The generalized Weyl module $W_{\lambda}$ is the cyclic $\mathcal{I}$-module with the cyclic vector $w_\lambda$ of $\fh$-weight $\lambda$ defined by the following relations for all $\al\in\Delta_+$:
\begin{align*}
 e_{\widehat{\sigma}(-\alpha+\delta)+r \delta} w_\la&=0,\ \al\in\Delta_+,\, r \geq 0; \\
e_{\widehat{\sigma}(\al)}^{\langle -\lambda_-, \al^\vee \rangle+1}  w_\lambda&=0,\\
 \fh\T z\bC[z] w_\lambda&=0.
\end{align*}
The definition of global generalized Weyl module $\mathbb{W}_{\lambda}$ is obtained by omitting the last relation.
\end{dfn}
\begin{rem}
If $\sigma$ is the identity element (i.e. $\la\in P_-$), then the generalized Weyl module $W_{\lambda}$ is isomorphic to the (local) Weyl module $W(w_0\lambda)$ (see \cite{FeMa}).
\end{rem}

\subsection{Two families of modules}\label{Relations}
In this subsection, we define two families of Iwahori algebra modules which play crucial role in this paper.

Suppose, as above, that $\lambda\in P$ and $\sigma\in W$ satisfy $\lambda=\sigma(\lambda_-)$. We define representations
$U_{\lambda}$ and ${\mathbb U}_{\lambda}$ as follows (see \cite{FKM}).
The module $U_{\lambda}$ is the cyclic $\mathcal{I}$-module with cyclic vector $u_{\lambda}$ of $\fh$-weight
$\lambda$
subject to the relations:
\begin{align*}
(e_{\widehat{\sigma}(-\alpha+\delta)+r \delta}) u_{\la}&=0,\ \al\in\Delta_+,\, r \geq 0, \\
 (f_{\sigma(\al)}\T z)^{-\bra \la_-,\al^\vee\ket+1} u_{\la}&=0,\ \al\in\Delta_+,\,
 \sigma\al\in\Delta_-,\\
(e_{\sigma(\al)}\T 1)^{-\bra \la_-,\al^\vee\ket} u_{\la}&=0,\ \al\in\Delta_+,\, \sigma\al\in\Delta_+,\\
\fh\T z\bC[z] u_{\la}&=0.
\end{align*}
The definition of the $\mathcal{I}$-module ${\mathbb U}_{\la}$ differs from the definition of the
$U_\la$ by removing the last relation.

In the similar way we define modules $D_{\la}$ and ${\mathbb D}_{\la}$ as follows.
$D_{\la}$ is the cyclic $\mathcal{I}$-module with cyclic vector $d_{\la}$ of $\fh$-weight
$\la$
subject to the relations:
\begin{align*}
(e_{\widehat{\sigma}(-\alpha+\delta)+r \delta}) d_{\la}&=0,\ \al\in\Delta_+,\,r \geq 0, \\
 (f_{\sigma(\al)}\T z)^{-\bra \la_-,\al^\vee\ket} d_{\la}&=0,\ \al\in\Delta_+,\,
 \sigma\al\in\Delta_-,\\
(e_{\sigma(\al)}\T 1)^{-\bra \la_-,\al^\vee\ket+1} d_{\la}&=0,\ \al\in\Delta_+,\, \sigma\al\in\Delta_+,\\
\fh\T z\Bbbk[z] d_{\la}&=0.
\end{align*}
The definition of the $\mathcal{I}$-module ${\mathbb D}_{\la}$ differs from the definition of the
$D_{\la}$ by removing the last relation.

\begin{rem}
It is easy to see that the $\mathcal{I}$-modules $U_\la$, ${\mathbb U}_{\la}$, $D_\la$, and ${\mathbb D}_{\la}$ (and $W_\la$, $\mathbb{W}_\la$) depend only on the weight 
$\la$, i.e., they do not depend on the choice of $\sigma\in W$ such that $\lambda=\sigma(\lambda_-)$. 
\end{rem}

\begin{rem}
The modules $U_{\la}$ and ${\mathbb U}_{\la}$ were introduced and studied in \cite{FKM}. In particular, it was shown that their characters are computed by the $t=\infty$ specializations of nonsymmetric Macdonald polynomials. See also \cite{Ka,NS,NNS} for related results on the $t=\infty$ specializations.

The modules $D_{\la}$ 
are analogues of affine Demazure modules (and coincide with these in type ADE, by \cite{FL}). We show below that the characters of $D_\la$ and ${\mathbb D}_{\la}$ are computed via the $t=0$ specializations of the nonsymmetric Macdonald polynomials in general.
\end{rem}

\subsection{Generalized Weyl modules with characteristic}\label{gWc}

Suppose $\la\in P$ and $\mu\in P_-$ satisfy $\la_- - \mu\in P_-$. We fix a reduced decomposition
\begin{equation}\label{pi}
t_\mu=\pi s_{j_1}\dots s_{j_l},\quad \pi\in\Pi,\ j_1,\dotsc,j_l\in I\cup\{0\},\ l=\ell(t_\mu)
\end{equation}
in the extended affine Weyl group $W^{a\vee}=P\rtimes W$. 
We consider the affine coroots $\beta_1^\vee,\dots,\beta_l^\vee$ defined \eqref{beta-u} for $u=t_\mu$ and this reduced expression.
Recall the decomposition $\beta_j^\vee=\bar\beta_j^\vee+ ({\rm deg}\, \beta_j^\vee )\delta^\vee$, where
$\bar \beta_j^\vee\in \Delta^\vee$, ${\rm deg} \, \beta_j^\vee \in \mathbb Z$. We note that $\bar\beta_j^\vee$ is always a negative coroot and ${\rm deg}\,
\beta_j^\vee>0$, since (\ref{pi}) is a reduced expression and $\mu\in P_-$.

For $\al\in\Delta_+$ and $m=1,\dots,l$ we define
\begin{equation}\label{l}
l_{\al,m}=-\bra\la_-,\al^\vee\ket - |\{j:\ \bar\beta_j^\vee=-\al^\vee, 1\le j\le m\}|.
\end{equation}

\begin{dfn}[\cite{FeMa,FMO1}]
Choose any $\sigma\in W$ such that $\lambda=\sigma(\lambda_-)$. The generalized Weyl module with characteristics $W_{\la}(m)$ is the ${\mathcal I}$-module which is the quotient of $W_{\la}$ by the submodule generated by
\begin{equation}\label{charrel}
e_{{\widehat \sigma}(\al)}^{l_{\al,m}+1} w_{\la}, \ \al\in\Delta_+.
\end{equation}
Similarly, we define the generalized global Weyl module with characteristics $\mathbb W_{\la}(m)$ as the ${\mathcal I}$-module  quotient of $\mathbb W_{\la}$ by the submodule generated by (\ref{charrel}) inside $\mathbb W_{\la}$.
\end{dfn}

\begin{rem}
The modules $W_\lambda(m)$ and $\mathbb W_\la(m)$ do not depend on the choice of $\sigma$, but they do depend on $\mu$ and the reduced decomposition of $t_\mu$.
\end{rem}

Now recall the  algebra $\mathcal{A}_{\la,m}$, which is
the quotient of the universal enveloping algebra $\U(z\fh[z])$ by the ideal annihilating the cyclic vector $w_{\la}\in \W_{\la}(m)$. 
The $\la$ weight space of the generalized global Weyl module with characteristics $\W_{\la}(m)$ is naturally
isomorphic to the algebra $\mathcal{A}_{\la,m}$ as a graded vector space \cite{FeMa,FMO1}.

For any $f\in \mathcal{A}_{\la,m}$ and $x\in \U(\mathcal{I})$ the map $xw_{\la}\mapsto xfw_{\la}$ is a well-defined $\mathcal{I}$-module endomorphism 
of $\W_{\la}(m)$. 
Thus, we may regard $\W_{\la}(m)$ as an $(\mathcal{I},\mathcal{A}_{\la,m})$-bimodule.

We need the following theorem (see \cite[Theorem 3.22]{FKM} and \cite[Theorem 3.17]{FMO1}):

\begin{thm} \label{characters}
Assume $\la,\mu\in P$, $\sigma\in W$, and $\beta_j^\vee\in\Delta^{a\vee}$ are as above.
\begin{itemize}
\item One has an isomorphism of graded algebras
\begin{equation}\label{Ah}
\mathcal{A}_{\la,m}:=\Bbbk[h_{\sigma(\alpha_j)}\otimes z^k]|_{j\in I,\, k=1,\dots,-\langle \lambda_-,\alpha_j^\vee \rangle-|\{i=1,\dots,m\,:\, -\bar \beta_i=\alpha_j^\vee\}|}.
\end{equation}
\item The right action of $\mathcal{A}_{\la,m}$ on $\W_{\sigma(\la_-)}(m)$ is free.
\item Let $\omega(m)$ be the sum of $\omega_j$ such that $-\bar\beta_i=\al_{j}^\vee$ is simple ($1 \le i \le m$). Then
\[
\ch \, \W_{\la}(m)=\frac{\ch \, W_{\la}(m)}{(q)_{\la_- + \omega(m)}}.
\]
\end{itemize}
\end{thm}

\begin{rem}\label{GlobalSpecialization}
Consider the one-dimensional $\mathcal{A}_{\la,m}$ module $\Bbbk_0$ such that 
the elements $h_{\sigma(\alpha_j)}\otimes z^k$, $j\in I$, $k\geq 1$, act by zero. Then
\[ W_{\la}(m)\simeq  \W_{\la}(m)\otimes_{\mathcal{A}_{\la,m}}\Bbbk_0.\]
\end {rem}

\subsection{Identifying modules}\label{NM}

\begin{prop}\label{DW}
For any $\lambda\in P$, and with $\sigma=\sigma_\la$, $\mu=\la_-$, and $\beta_j^\vee$ defined by the decomposition \eqref{reddecV}, we have isomorphisms:
$$D_{\la} \simeq W_{\la}(\ell(\sigma)),\qquad
\mathbb{D}_{\la} \simeq \W_{\la}(\ell(\sigma)).$$
\end{prop}

\begin{proof}
We have to prove that the defining relations \eqref{charrel} of $W_{\la}(\ell(\sigma))$ coincide
with the defining relations of $D_{\la}$. We set $r:=\ell(\sigma)$,
which is the cardinality of the set $\Delta_+\cap\sigma^{-1} \Delta_-$.

It suffices to show that $\{-\bar\beta_1^\vee,\dots,-\bar\beta_r^\vee\}=\Delta_+^\vee\cap\sigma^{-1} \Delta_-^\vee$. By definition, for $k=1,\dots,r$ we
have
\begin{eqnarray*}
\beta_k^\vee & = & s_{i_M}\dots s_{i_{r+1}} s_{\pi^{-1}i_r}\dots s_{\pi^{-1}i_{k+1}} \al_{\pi^{-1}i_k}^\vee\\
&= & t_{-\la_-} \pi s_{\pi^{-1}i_1}\dots s_{\pi^{-1}i_{k-1}} s_{\pi^{-1}i_k}  \al_{\pi^{-1}i_k}^\vee \\
&= & t_{-\la_-} s_{i_1}\dots s_{i_{k-1}} (-\al_{i_k}^\vee).
\end{eqnarray*}
The action of $t_{-\la_-}$ on $\Delta^{a\vee}$ changes only the coefficient of $\delta^\vee$, i.e., it preserves the classical part of any affine coroot). Therefore, $-\bar\beta_1^\vee,\dots,-\bar\beta_r^\vee$ are exactly the positive coroots which are mapped to negative coroots by $\sigma$. Comparison of defining relations then yields the assertion.
\end{proof}

We define $\mathcal{A}^D_{\la}:=\mathcal{A}_{\la,\ell(\sigma_\la)}$
 with respect to the reduced decomposition \eqref{reddecV}. (In fact, any choice of $\sigma\in W$ such that $\la=\sigma(\la_-)$ produces the same algebra $\mathcal{A}_{\la,\sigma}$.)

\begin{cor}\label{End=End}
The algebra $\mathcal{A}^D_{\la}$ acts freely on the modules $\mathbb{D}_{\la}$ and $\mathbb{U}_{-\la}$ and 
\begin{gather*}
\ch \, {\mathbb U}_{-\la} = \ch \, U_{-\la}\,\ch\,\mathcal{A}^D_{\la},\qquad
\ch \, {\mathbb D}_{\la} = \ch \, D_{\la}\,\ch\,\mathcal{A}^D_{\la}.
\end{gather*}
Moreover
\begin{equation}\label{EndomorphismsUD}
    \mathrm{End}_{\mathcal{I}}({\mathbb U}_{-\la})\simeq \mathrm{End}_{\mathcal{I}}({\mathbb D}_{\la})\simeq \mathcal{A}^D_{\la}.
\end{equation}
\end{cor}
\begin{proof}
The algebra $\mathcal{A}^D_{\la}$ acts freely on  $\mathbb{D}_{\la}$ thanks to Proposition \ref{DW} and Theorem \ref{characters}. The algebra $\mathcal{A}^D_{\la}$ acts freely on  $\mathbb{U}_{-\la}$ thanks to \cite[Proposition 4.20]{FMO2}. The second claim follows from Remark \ref{GlobalSpecialization}. Finally, highest weight algebras coincide with the algebras of endomorphisms of the corresponding Iwahori modules. This implies \eqref{EndomorphismsUD}.
\end{proof}

In the next theorem we compute the character of the modules $D_\lambda$, extending the interpretation of specialized Macdonald polynomials as affine Demazure characters \cite{Sa,I} to non-simply-laced untwisted affine types (see also \cite{CI,LNSSS2}).

\begin{thm}\label{Dcharacter}
The character of $D_\lambda$ is equal to $E_\lambda(x, q, 0)$.
\end{thm}
\begin{proof}
 Recall the reduced decomposition \eqref{reddecV} of $t_{\lambda_-}$.
For each $0 \le r < M=\ell(t_{\la_-})$, we consider the subset ${\mathcal{QB}}_{\id,\la_-}(r)$ of alcove paths $p_J$ in ${\mathcal{QB}}(\id; t_{\la_-})$ for which $J\subset\{r+1,\dotsc,M\}$, i.e., we allow only the reflections through
$\beta_{r+1}^\vee,\dots,\beta_M^\vee.$

Let $r=\ell(\sigma_\la)$.
By \cite[Corollary 3.21]{FKM}, we have
 \[
\ch \, W_{\la}(r)=\sum_{p\in {\mathcal{QB}}_{id,{\la_-}}(r)} x^{{\rm wt}({\rm end} (p))} q^{\deg(\qwt(p))}.
\]
The left hand side is equal to the character of $D_\lambda$ by Proposition \ref{DW} and the right hand side is equal to $E_\la(x,q,0)$ by \eqref{tm}.
\end{proof}

\begin{rem}
One has an analogue of Proposition \ref{DW} for the modules $U_\la$ and ${\mathbb U}_{\la}$. More precisely, it was shown in \cite{FKM} that these 
modules are also isomorphic to certain generalized Weyl modules with characteristics.
\end{rem}

\subsection{Right modules}
In this subsection we define right modules $U_{\la}^o$ and ${\mathbb U}_{\la}^o$.

For any $\la\in P$ and $\sigma\in W$ such that $\la=\sigma(\la_-)$, the right modules 
$U_{\la}^o$ and ${\mathbb U}_{\la}^o$ are defined (independently of the choice of $\sigma$) as follows.
$U_{\la}^o$ is the cyclic right $\mathcal{I}$-module with cyclic vector $u_{\la}^o$ of $\fh$-weight
$\la$
subject to the relations:
\begin{align*}
u_{\la}^o e_{\widehat{\sigma}(\alpha)+r \delta} &=0,\ \al\in\Delta_+,\, r \geq 0, \\
 u_{\la}^o (f_{\sigma(\al)}\T z)^{-\bra \la_-,\al^\vee\ket+1}&=0,\ \al\in\Delta_-,\,
 \sigma\al\in\Delta_-,\\
u_{\la}^o(e_{\sigma(\al)}\T 1)^{-\bra \la_-,\al^\vee\ket} &=0,\ \al\in\Delta_-,\, \sigma\al\in\Delta_+,\\
u_{\la}^o\fh\T z\Bbbk[z] &=0.
\end{align*}
As before the definition of the $\mathcal{I}$-module ${\mathbb U}^o_{\la}$ differs from the definition of the
$U^o_{\la}$ by removing the last relation.


One has the following proposition (see Proposition 4.16, Corollary 4.17, and Corollary 4.18 in \cite{FMO2}):

\begin{prop}\label{UUo}
The modules $U^o_\la$ and ${\mathbb U}^o_{\la}$ enjoy the following properties:
\begin{itemize}
\item 
One has an isomorphism of graded vector spaces $U_{\la}^o\simeq U_{-\la}$, with the $\mathcal{I}$-module structures related by the map $x\mapsto -x$ on $\mathcal{I}$.
\item 
The module ${\mathbb U}_{\la}^o$ has a structure of   $(\mathcal{A}^D_{\la},\mathcal{I})$-bimodule and is free as a  left $\mathcal{A}^D_{\la}$-module.
\item For any $\lambda \in P$, one has
$\ch\, U_{\lambda}^o=E_{\lambda}(y,q^{-1},\infty).$
\end{itemize}
\end{prop}

\section{Iwahori modules and stratified categories}\label{IwahoriStratified}
\subsection{Standard and costandard objects}
Let $\fC$ be the category of $\mathcal{I}$-modules $M$ subject to the following conditions: 
\begin{itemize}
\item $M=\bigoplus_{n\in\bZ} M_{n}$ is a graded $\mathcal{I}$-module such that $M_{n}=0$ for $n$ sufficiently large;
\item each $M_{n}$ admits an integral $\fh$-weight space decomposition 
$M_{n}=\bigoplus_{\nu\in P} M_{n}(\nu)$ such that $\dim M_{n}(\nu)<\infty$ for all $\nu\in P$;
\item $M$ is cogenerated by a subspace $\overline{M}$ admitting $\fh$-weight decomposition  $\overline{M}=\bigoplus_{\nu\in P} \overline{M}(\nu)$ such that $\dim \overline{M}(\nu)<\infty$ for all $\nu\in P$.
\end{itemize}
The simple objects in $\fC$ are one-dimensional spaces of the form $L_{(\la,k)}$, $\la\in P$, $k\in \bZ$, where $\la$ is responsible for the $\fh$-weight and $k$ fixes the $z$-degree. As before, we write $L_\la$ to mean $L_{(\la,0)}$ and adopt this convention for all families of modules indexed by $P\times\bZ$ considered below.

We also have a similar definition for the ``dual'' category $\fC^\star$, where the last property above is replaced by 
\begin{itemize}
\item $M$ is generated by a subspace $\overline{M}$ admitting an $\fh$-weight decomposition $\overline{M}=\bigoplus_{\nu\in P} \overline{M}(\nu)$ such that $\dim \overline{M}(\nu)<\infty$ for all $\nu\in P$.
\end{itemize}
Slightly abusing the notation, 
we use the same notation $L_{(\la,k)}$ to denote the (same) simple objects.

The categories $\fC$ and $\fC^\star$ agree with those of Section~\ref{BimoduleFunctions}, with the space $\overline{M}$ given by $\mathrm{ker}\,\mathfrak{r}$ and its dual, respectively. Here $\fg(0)=\fn^+\oplus\fh$, $\fg(m)=\fs\otimes z^m$ for $m>0$, $\fr=\fn^+\oplus \fs\otimes z\Bbbk[z]$, and $\fg_0=\fh$.

We have the contravariant equivalence $M\mapsto M^\vee$ between $\fC$ and $\fC^\star$ given by the restricted dual.

Recall the orders $\preceq$ and $\preceq^\vee$ on $P$ from Definition~\ref{ChOrder}. For each $\la\in P$, we consider the subcategory $\fC_{\preceq \lambda}$ of the category $\fC$ consisting of modules $M$
such that $M_{n}(\nu)\ne 0$ for some $n$ implies $\nu\preceq\la$. We also define subcategories $\fC_{\preceq^\vee \lambda}$ of $\fC$ by replacing $\preceq$ by $\preceq^\vee$ in this definition.

In the same way, we define the subcategories $\fC_{\preceq^\vee  \lambda}^\vee$ and $\fC_{\preceq\la}^\vee$ of $\fC^\star$.

\begin{lem}\label{DualitySubcategories} The following holds:
 \begin{itemize}
     \item   $M \in \fC_{\preceq \lambda}$ if and only if $M^\vee \in \fC_{\preceq^\vee - \lambda}^\vee$;
     \item   $M \in \fC_{\preceq \lambda}^\vee$ if and only if $M^\vee \in \fC_{\preceq^\vee - \lambda}$.
\end{itemize}  
\end{lem}
\begin{proof}
One checks that $\lambda\mapsto -\lambda$ is an order-preserving bijection between $(P,\preceq)$ and $(P,\preceq^\vee)$.
\end{proof}

\begin{rem}
The category $\fC$ has enough injective objects, but does not contain projective objects. On the contrary, 
 the category $\fC^\star$ has enough projective objects, but does not have injective objects.
\end{rem}

Recall the standard $\Delta_\la$, proper standard $\overline{\Delta}_\la$, costandard $\nabla_\la$ and proper costandard $\overline{\nabla}_\la$ objects from Definition \ref{StandardCostandard}. The (co)standard objects are guaranteed to exist by Remark~\ref{StandardExist}. We will see that the proper (co)standard objects exist and are finite-dimensional in Theorem~\ref{DUStCost} below.

\begin{example}
The modules $\overline{\Delta}_\la$, $\nabla_\la$, $\overline{\nabla}_\la$, $\Bbbk[{\bf I}]$ are elements of $\fC$. The modules  $\Delta_\la$, $\overline{\Delta}_\la$,  $\overline{\nabla}_\la$ are elements of $\fC^\star$.
\end{example}

\begin{lem}
The contravariant dualization functor $\vee: (\fC^\star)^{op} \rightarrow \fC$ and its restriction to subcategories
$\vee:(\fC_{\preceq^{\vee} \lambda}^{\vee})^{op} \rightarrow \fC_{\preceq\lambda}$ sends the projective objects to injectives and (proper) standard objects to (proper) costandard objects. In particular:
\[
\mathbb{P}_{\lambda}^{\vee} = \mathbb{I}_{\lambda}, \qquad \Delta_{\lambda}^{\vee} = \nabla_{\lambda}, \qquad (\overline{\Delta}_{\lambda})^{\vee} = \overline{\nabla}_{\lambda}.
\] 
\end{lem}

\begin{rem}
The modules ${\Delta}_\la$ do not belong to $\fC$ because they contain vectors of arbitrary large $z$-degrees.
The modules ${\nabla}_\la$ and  $\Bbbk[{\bf I}]$ do not belong to $\fC^\star$ because they contains vectors of arbitrary small (negative) $z$-degrees.
\end{rem}

\begin{thm}\label{DUStCost}
 One has the isomorphism of modules of the Iwahori algebra: 
 \[
{\mathbb D}_\la=\Delta_\la,\qquad {D}_\la=\overline{\Delta}_\la,\qquad
{\mathbb U}^\vee_{-\la}=\nabla_\la,\qquad {U}^\vee_{-\la}=\overline{\nabla}_\la.
 \]
\end{thm}
\begin{proof}
 Recall that the character of $D_\la$ is equal to $E_\la(x,q,0)$ (see Theorem \ref{Dcharacter}).
 In addition, for the long element $w_0\in W$, we have
 \[
\ch\, U_{-\la}^\vee =  w_0E_{-w_0\la}(x^{-1},q,\infty) = E_{\la}(x,q,\infty),
 \]
where the first equality is proved in \cite{FKM} and the second equality holds because $-w_0$ is the second order diagram automorphism. 
 Since all the monomials $x^\mu$ showing up in  $E_{\la}(x,q,t)$
satisfy $\mu\preceq \la$, 
we obtain that $\mathbb{D}_{\lambda},D_{\lambda},{U}^\vee_{-\la}\in \fC^\star_{\preceq \la}$,
$\mathbb{U}^\vee_{-\lambda},D_{\lambda},{U}^\vee_{-\la}\in \fC_{\preceq \la}$.

Let $M \in \fC^\star_{\preceq\la}$ and suppose
    \[\psi:L_\lambda \rightarrow M \]
    is a nonzero morphism. Let $v$ be a nonzero element of $L_\lambda$. Then $\psi(v)$ is an element of weight $\lambda$.
    Choose any $\sigma\in W$ such that $\sigma(\la_-)=\la$. Then we clearly have:
\begin{align}
\wt\big(e_{\widehat{\sigma}(-\alpha+\delta)+r \delta} v\big)&\succ\lambda,\ \al\in\Delta_+,\,r \geq 0, \notag\\
\label{highWeightrelation}
\wt\big( (f_{\sigma(\al)}\T z)^{-\bra \la_-,\al^\vee\ket} v\big)&\succ\lambda,\ \al\in\Delta_+,\,
 \sigma\al\in\Delta_-,\\
\wt\big((e_{\sigma(\al)}\T 1)^{-\bra \la_-,\al^\vee\ket+1} v\big)&\succ\lambda,\ \al\in\Delta_+,\, \sigma\al\in\Delta_+,\notag
\end{align}
in any case where the vector inside $\wt$ is nonzero,
where $\wt$ denotes the weight of an $\fh$-weight vector.
Therefore
\begin{align*}
e_{\widehat{\sigma}(-\alpha+\delta)+r \delta} v&=0,\ \al\in\Delta_+,\,r \geq 0, \\
 (f_{\sigma(\al)}\T z)^{-\bra \la_-,\al^\vee\ket} v&=0,\ \al\in\Delta_+,\,
 \sigma\al\in\Delta_-,\\
(e_{\sigma(\al)}\T 1)^{-\bra \la_-,\al^\vee\ket+1} v&=0,\ \al\in\Delta_+,\, \sigma\al\in\Delta_+.
\end{align*}
Thus there exists a map 
\[\varphi:\mathbb{D}_\lambda \rightarrow M, ~d_\lambda \mapsto \psi(v).\]

If the $\la$-weight space in $M$ is one-dimensional then we also have
\[z\fh[z]\psi(v)=0.\] 
Therefore we have a map
\[\varphi:{D}_\lambda \rightarrow M, ~d_\lambda \mapsto \psi(v)\]
extending the map $\psi$. Thus $D_\la =\overline\Delta_\la$.

By the same arguments we have that $\mathbb{U_{-\la}}$ is the projective cover of $L_{-\la}$ in $\fC^\star_{\preceq \la}$. Thus using Lemma \ref{DualitySubcategories} we have ${\mathbb U}^\vee_{-\la}=\nabla_\la$. Similarly $U^\vee_{-\la} =\overline\nabla_\la$.
\end{proof}

\begin{dfn}
    For pairwise incomparable weights $\lambda_1, \dots, \lambda_k$, let $\fC_{\preceq\lambda_1, \dots, \lambda_k}$ be the full subcategory of the category $\fC$ consisting of modules $M$
such that $M_n(\nu)\ne 0$ implies $\nu\preceq \la_i$ for some $i=1,\dots,k$.
Similarly, we consider subcategories $\fC^\star_{\preceq \lambda_1, \dots, \lambda_k}$ of $\fC^\star$.
\end{dfn}

\begin{prop}
    For $i=1,\dotsc,k$, $\Delta_{\lambda_i}$ is projective in $\fC^\star_{\preceq \lambda_1, \dots, \lambda_k}$ and $\nabla_{\lambda_i}$ is injective in $\fC_{\preceq \lambda_1, \dots, \lambda_k}$. 
\end{prop}
\begin{proof}
  The first claim follows from equations \eqref{highWeightrelation}.  By the defining relations of $\mathbb{U}_{\la_i}$ we have that this module is projective in $\fC^\star_{\preceq^\vee \lambda_1, \dots, \lambda_k}$. Thus the second claim follows from Lemma \ref{DualitySubcategories}.
\end{proof}

\subsection{Main character identity}
Our goal now is to prove the following conjecture of \cite{FMO2}: 

\begin{thm}\label{conj}
 One has the equality
 \begin{equation}\label{BiCharacterEquality}
 \sum_{\la\in P} \ch\,\Delta_\la(x,q)\,\ch\,\overline\nabla_\la^{\vee o} (y,q) = \sum_{\la\in P} \ch\,\bP_\la(x,q)\,\ch\,L_\la(y,q)
 \end{equation}
 of well-defined series in $(\ZZ[[P\times P]])[[q]]$. 
\end{thm}

Recall the formulas from Section~\ref{curalg}:
\[
\ch\,\Delta_\la (x,q)=(q)_\la^{-1}E_\la(x,q,0),\quad
\ch\,\overline\nabla_\la^{\vee o} (y,q) = E_\la(y,q^{-1},\infty).
\]
We also have
\[
\ch\,\bP_\la(x,q) = \frac{x^\la}{(q;q)_\infty^n \prod_{\al\in\Delta_+} (1-x^\alpha)\prod_{\alpha\in\Delta}(q x^\alpha;q)_\infty},\quad \ch\,L_\la(y,q)=y^\la.
\]
Substituting these into \eqref{BiCharacterEquality} converts Theorem~\ref{conj} into the form of \cite[Conjecture 2]{FMO2}.

\begin{prop}\label{PropWellDef}
The sum
\begin{equation}\label{BiCharacterE}
\sum_{\mu\in P} \ch\, \Delta_\mu (x,q)\, \ch\,\overline\nabla_\mu^\vee (y,q) = \sum_{\mu\in P} (q)_\mu^{-1}E_\mu (x,q,0) E_\mu (y^{-1},q^{-1},\infty).
\end{equation}
is well-defined in $(\ZZ[[P\times P]])[[q]]$.
\end{prop}
\begin{proof}
We will show that for arbitrary weights $\la,\nu\in P$ the coefficient of $x^\la y^{-\nu}$ in
$$\sum_{\mu\in P} (q)_\mu^{-1}E_\mu (x,q,0) E_\mu (y^{-1},q^{-1},\infty)$$ 
is a well-defined series in $q$. 
It suffices to prove that for any $M\in \mathbb{Z}_{\geq 0}$ there exists only finite number of weights $\mu\in P$ such that $(q)_\mu^{-1}E_\mu (x,q,0) E_\mu (y^{-1},q^{-1},\infty)$ contains a term $q^Mx^\la y^{-\nu}$. We show that there exists only finite number of weights $\mu$ such that the lowest $q$-degree terms in the coefficients in front of $x^\lambda$ in $E_\mu (x,q,0)$ and  in front of $y^{-\nu}$ in $E_\mu (y^{-1},q^{-1},\infty)$ both are less than or equal to $M$.

The coefficient in front of $x^\lambda$ in $E_\mu (x,q,0)$ has a chance to have a term with $q$-degree not exceeding  $M$ only if $\lambda-\mu$ admits a decomposition into a sum of roots containing no more than $M$ negative roots. 
In fact, $E_\mu (x,q,0)$ is the character of the Iwahori algebra module $\Delta_\mu$. The module $\Delta_\mu$ is cyclic and the Iwahori algebra is generated by $\fn_+$ and $z\fn_-$. Hence whenever one applies to the cyclic vector of $\Delta_\mu$ several operators from $z\fn_-$, one obtains a vector of $z$-degree equal to the number of operators applied. 
Similarly, one shows that the coefficient in front of $y^{-\nu}$ in $E_\mu (y^{-1},q^{-1},\infty)$ has a chance to have a term with $q$ degree not exceeding $M$ only if $-\nu-(-\mu)$ admits a decomposition into a sum of roots containing no more than $M$ negative roots
(recall that $E_\mu (y^{-1},q^{-1},\infty)$ is the character of the cyclic module $U_{-\mu}$, whose highest weight is equal to $-\mu$). 

So we are looking for the weight $\mu$ such that for fixed $\la$ and $\nu$ the following conditions hold true:
\begin{itemize}
\item   $\lambda-\mu$ admits a decomposition into a sum of roots containing no more than $M$ negative roots,
\item $\mu-\nu$ admits a decomposition into a sum of roots containing no more than $M$ negative roots.
\end{itemize}
Given a pair of such decompositions, we sum them up and obtain a decomposition of $\lambda-\nu$ into a sum of roots containing at most $2M$ negative roots.
The number of such decompostions is finite. Since $\mu-\nu$ is obtained as a sum of several roots from a decomposition of $\lambda-\mu$ as above, we conclude that the number of $\mu$ satisfying the two conditions above is finite.
Hence  
for any $M$ there exists only finite number of weights $\mu\in P$ such that the product $(q)_\mu^{-1}E_\mu(x,q,0) E_\mu (y^{-1},q^{-1},\infty)$ contains a term $q^Mx^\la y^{-\nu}$. 
\end{proof}

\begin{rem}
One can also derive the properties of specialized nonsymmetric Macdonald polynomials used in the preceding proof from the Ram-Yip formula \cite{RY} and its $t=0,\infty$ specialized forms \cite{OS} (see \eqref{E0}).
\end{rem}

Theorem \ref{conj} is an immediate consequence of the following (by considering the coefficient of $y^\la$ on both sides of \eqref{BiCharacterEquality}):

\begin{cor}\label{CorPIntoE}
For any $\la\in P$, we have
\begin{align}\label{PIntoE}
\ch\,\bP_\la = \sum_{\mu\in P} m_{\la,\mu}(q)\frac{E_\mu(x,q,0)}{(q)_\mu}
\end{align}
where $m_{\la,\mu}(q)$ is the coefficient of $y^\la$ in $E_\mu(y,q^{-1},\infty)$.
\end{cor}

\begin{proof}
First, the convergence of the right-hand side of \eqref{PIntoE} is guaranteed by Proposition~\ref{PropWellDef}. The coefficients $m_{\la,\mu}(q)$ are given in terms of the Cherednik inner product as follows:
\begin{align}
m_{\la,\mu}(q) &= \langle \ch\,\bP_\lambda, E_\mu(x^{-1},q^{-1},\infty)\rangle_0.
\end{align}
Now consider the difference of both sides of \eqref{PIntoE}:
$$ Q_\la := \ch\,\bP_\la-\sum_{\mu\in P} (q)_\mu^{-1}E_\mu(x,q,0)m_{\la,\mu}(q). $$
This element satisfies $\langle Q_\la, E_\mu(x^{-1},q^{-1},\infty)\rangle_0 =0$ and hence $\langle Q_\la, x^{-\mu}\rangle_0=0$ for all $\mu\in P$. 

We claim that $Q_\la=0$. If not, write $Q_\la=\sum_{k\ge 0}Q_{\la,k}q^k$ in $(\ZZ[[P]])[[q]]$. By the formula for $\ch\,\bP_\la$ and the property of $E_\mu(x^{-1},q^{-1},\infty)$ used in the proof of Proposition~\ref{PropWellDef}, one sees that for any fixed $k\ge 0$, there exist finitely many distinct weights $\nu_1,\dotsc,\nu_s\in P$ such that $Q_{\la,k}\in\ZZ[[P]]$ contains only monomials $x^\eta$ satisfying $\eta\in\bigcup_{i=1}^s (\nu_i+Q_+)$. Now, suppose $k\ge 0$ is minimal such that $Q_{\la,k}\neq 0$. Choose $\nu\in P$ such that $x^\nu$ occurs in $Q_{\la,k}$ with nonzero coefficient and $x^{\eta}$ does not occur in $Q_{\la,k}$ for all $\eta\in\nu-Q_+$. We arrive the contradiction $\langle Q_{\la}, x^{-\nu}\rangle_0\neq 0$. Indeed, since 
\begin{align*}
\langle Q_{\la},x^{-\nu} \rangle_0 \equiv \langle Q_{\la,k}x^{-\nu}\prod_{\alpha\in\Delta_+} (1-x^\alpha)\rangle \pmod{q},
\end{align*}
we see that the coefficient of $q^0
$ in $\langle Q_\la, x^{-\nu}\rangle_0$ is nonzero.
\end{proof}



\subsection{Bimodule of functions}
Recall that irreducible objects in $\fC$ and in $\fC^\star$ are one-dimensional and are denoted by $L_{(\lambda,k)}$.
Let $\Lambda:=\{(\lambda,k) \colon \lambda\in P, k\in\bZ\}$, $\Theta=P$.
Consider the projection $\rho:\Lambda\to P$ that forgets the grading. 
Recall that we defined two partial orders on the set of weights $P$ in Definition~\ref{ChOrder}. The orders are  denoted by $\preceq$ and $\preceq^{\vee}$; we call  them the Cherednik and the dual Cherednik orderings. 

We have the grading shift functor $\cdot[k]$ such that for each graded $\mathcal{I}$-module $M$, the module $M[k]$ is isomorphic to $M$ as an $\mathcal{I}$-module and the $z$-degree of all its elements is increased by $k$. In particular we have:
\[\ch \,M[k]=q^k\ch \,M.\]
Let $M,N$ be $\mathcal{I}$-modules. We denote:
\begin{equation}\mathrm{ext}(M,N)=\sum_{k=0}^\infty \sum_{i=0}^\infty(-1)^i\dim\Ext^i(M[k],N)q^k\label{GradedExt}
\end{equation}
whenever the right-hand side is defined. We observe that
\begin{equation}\label{q0ext}
\text{the coefficient of $q^0$ in $\mathrm{ext}(M,N)$ is $([M],[N])$.}
\end{equation}

In \cite[Appendix A]{FKM} the following was proved:
\begin{lem}\label{ExtCherednik}
For $[M],[N]\in K_0(\fC^\star)$, with $M\in D(\fC^\star)$ and $N\in\fC^\star$ finite-dimensional as in \eqref{eq::pairing::def}, we have
\[\mathrm{ext}(M,N)=  
 \langle  \ch(M),\ch(N^\vee)\rangle_0.\]
\end{lem}

\begin{cor}\label{ProjStand}
The sum $\sum_{(\mu,k)\in P\times\ZZ} ([\bP_\la],[\overline\nabla _{(\mu,k)}])[\Delta_{(\mu,k)}]$ is well-defined and equal to $[\bP_\la]$ in $K_0(\fC^\star)$.
\end{cor}

\begin{proof}
This is immediate from Corollary~\ref{CorPIntoE} and the identification of $K_0(\fC^\star)$ with $\ZZ[[P]][[q]]$.
\end{proof}



Now we are ready to prove our main results concerning $\mathcal{I}$-modules (for the Ext pairing in the simply-laced case see  \cite{ChKa,FKM}).

\begin{thm}\label{TheoremStr}
One has:
\begin{itemize}
    \item the category $\fC$ is $(P,\preceq)$-stratified;
    \item the category $\fC$ is $(P,\preceq^{\vee})$-stratified;
    \item the category $\fC^\star$ is $(P,\preceq)$-stratified;
    \item the category $\fC^\star$ is $(P,\preceq^{\vee})$-stratified.
\end{itemize}
\end{thm}
\begin{proof}
The proof of each item is absolutely the same, so we will explain just the third one.
In order to show that $\fC^\star$ is stratified it is enough to verify the conditions of Theorem~\ref{thm::characters::stratified}.
Indeed, $\fC^\star$ has enough projectives; standard modules $\Delta_{(\lambda,k)}$ remain projective in $\fC^\star_{\preceq \lambda_1,\ldots,\lambda_k}$ for incomparable weights $\{\lambda_1,\ldots,\lambda_k\}$, which implies condition $\rm{(i)}$. Proper costandard modules are finite-dimensional and hence belong to $\fC^\star$ (condition $\rm{(iii)}$). The second condition $\rm{(ii)}$ is verified in Corollary~\ref{ProjStand}. 
Condition \ref{eq::standard::orthogonal} is implied by \eqref{GradedExt} and Lemma \ref{ExtCherednik}. 
Finally, thanks to Theorem~\ref{DUStCost} we know the characters of the (proper) standard and (duals of) costandard modules:
\begin{align*}
\ch\,\overline\Delta_{(\lambda,k)} &= q^k E_\la(x,q,0),
&&\ch\,\Delta_{(\lambda,k)} = q^k (q)_\lambda^{-1} E_\la(x,q,0),\\
\ch\,\overline\nabla_{(\lambda,k)}^\vee &= q^{-k} E_\la(x^{-1},q^{-1},\infty),
&&\ch\,\nabla_{(\lambda,k)}^\vee = q^{-k} (q^{-1})_\lambda^{-1} E_\la(x^{-1},q^{-1},\infty).
\end{align*}
It is then immediate from \ref{EOrth0} that $\ch\,\Delta_{(\la,k)}$ and $\ch\,\overline\nabla_{(\la,k)}^\vee$ are orthonormal with respect to pairing given by taking the coefficient of $q^0$ in the specialized Cherednik inner product $\langle\,,\,\rangle_0$. Condition $\rm{(iv)}$ of Theorem~\ref{thm::characters::stratified} then follows from Lemma~\ref{ExtCherednik} and the observation \eqref{q0ext}. We conclude that the category $\fC^\star$ is $(P,\preceq)$-stratified.
\end{proof}

\begin{thm}
 The space of function $\Bbbk[\bfI]$ admits a filtration such that 
 \[
 {\rm gr}\left( \Bbbk[\bfI]\right) \simeq\bigoplus_{\la\in P} \nabla_\la\T_{\cA_\la} (\Delta_\la^\vee)^o. 
 \]
\end{thm}
\begin{proof}
In order to use Theorem \ref{FT} we need to check that all the conditions of sections \ref{progroup} and \ref{FandBi} are satisfied. For $\fg=\mathcal{I}\subset \mathfrak{s}[z]$ one has:
\begin{itemize}
\item $\fg_0=\fh$, $\fr=\fn_+\oplus \mathfrak{s}\T z\Bbbk[z]$, $\overline\fg = \fh[z]$;
\item $G_0=(\Bbbk^*)^{\rm{rk}\mathfrak{s}}$,
$G=\bfI$, $R=\exp(\fr)$ -- the unipotent Lie 
group of $\fr$.
\end{itemize}
Then all the conditions from section \ref{progroup} are satisfied (this is an easy check apart from condition \eqref{SCat}, which is Theorem \ref{TheoremStr}). 
Now assumption \eqref{abelian} is obvious and condition \eqref{kernel} comes from Theorem \ref{DUStCost} and Corollary \ref{End=End}.
\end{proof}

\subsection{Character expansions}

\subsubsection*{Proper standard into projective}
Since the category $\fC^\star$ has enough projectives, there exists for each $\lambda\in P$ an exact sequence of the form
\begin{align}\label{StProjRes}
\dotsm \to P^1\to P^0 \to \overline{\Delta}_\lambda\to 0
\end{align}
with each $P^i$ given by a direct sum of the projectives $\bP_{(\mu,k)}$. In fact, we can construct such resolutions by starting with the BGG resolution of the trivial module $L_0$ in terms of lowest-weight Verma modules for the affine Kac-Moody algebra associated with $\mathfrak{s}$:
\begin{equation}
\dots\rightarrow\bigoplus_{\substack{w \in W^a\\ l(w)=1}}\bP_{\varrho^a-w(\varrho^a)} \rightarrow \bP_0 \rightarrow L_0,
\end{equation}
where $\varrho^a$ is the sum of affine fundamental weights and $W^a=W\ltimes Q^\vee$ is the affine Weyl group.
Restricted to the Iwahori subalgebra, this gives a projective resolution of  $L_0$ in the category $\fC^\star$. We can then take the tensor product with $\overline{\Delta}_\lambda$, which is finite-dimensional, to obtain a projective resolution \eqref{StProjRes}. Moreover, this resolution inherits the following properties from the BGG resolution:
\begin{itemize}
\item each $P^i$ is a finite direct sum of indecomposable projective modules;
\item for any fixed $k\ge 0$, the complex $P^\udot$ involves only  finitely many projectives of the form $\bP_{(\mu,k)}$ for various $\mu\in P$;
\item for any fixed $\mu\in P$, the complex $P^\udot$ involves only  finitely many projectives of the form $\bP_{(\mu,k)}$ for $k\ge 0$.
\end{itemize}

One then has a character expansion
\begin{align}
\label{StIntoProj}
\mathrm{ch}\,\overline{\Delta}_\lambda &= \sum_{i=0}^\infty (-1)^i \mathrm{ch}\,P^i
= \sum_{k=0}^\infty \sum_{\mu\in P}a_{\mu,k}^\lambda\mathrm{ch}\,\bP_{\mu}q^k\in \mathbb{Z}[P][[q]]
\end{align}
for certain integers $a_{\mu,k}^\lambda\in\mathbb{Z}$ and in which each inner sum (over $\mu\in P$ for fixed $k\ge 0$) contains finitely many nonzero terms. The coefficients are determined by the equality
$$\sum_{k=0}^\infty a_{\mu,k}^\lambda q^k = \mathrm{ext}([\overline{\Delta}_\lambda], [L_{\mu}]) = \langle E_\lambda(x,q,0), x^{-\mu} \rangle_0$$
which shows that $a_{\mu,k}^\lambda = 0$ unless $-\mu\succeq\lambda$ and also exhibits the property that $\sum_{k=0}^\infty a^\lambda_{\mu,k}q^k$ is a polynomial, i.e., has finitely many terms.

The combinatorial shadow of the BGG resolution above is Macdonald's identity \cite{Mac-eta}, which makes fully explicit the coefficients $a_k(x)\in\mathbb{Z}[Q]$ in the expansion
\begin{align}
\label{MacId}
(q;q)_\infty^{\mathrm{rk}(\mathfrak{s})}\prod_{\alpha\in\Delta_+}(1-x^\alpha)\prod_{\alpha\in\Delta}(qx^\alpha;q)_\infty &= \sum_{k=0}^\infty a_k(x) q^k.
\end{align}
Writing $a_k(x)=\sum_{\beta} a_{\beta,k}x^\beta \in\mathbb{Z}[Q]$, we have
\begin{align}
\label{Mac1}
1 &= \sum_{k=0}^\infty\frac{a_k(x)}{(q;q)_\infty^{\mathrm{rk}(\mathfrak{s})}\prod_{\alpha\in\Delta_+}(1-x^\alpha)\prod_{\alpha\in\Delta}(qx^\alpha;q)_\infty}\\
\notag
&= \sum_{k=0}^\infty \sum_{\beta\in Q} a_{\beta,k} \mathrm{ch}\,\bP_\beta q^k,
\end{align}
so that $a^0_{\beta,k}=a_{\beta,k}$ and $a^0_{\mu,k}=0$ for $\mu\notin Q$.
Macdonald's identity exhibits the two forms of finiteness described above:
\begin{itemize}
\item for each $k\ge 0$, one has $a_{\beta,k}=0$ for all but finitely many $\beta\in Q$;
\item for each $\beta\in Q$, one has $a_{\beta,k}=0$ for all but finitely many $k\ge 0$, i.e., each indecomposable projective $\bP_\beta$ for $\beta\in Q$ occurs with finite graded multiplicity.
\end{itemize}
Multiplying on both sides of \eqref{Mac1}, one sees that any polynomial in $\mathbb{Z}[q][P]$ admits an expansion $\sum_{k=0}^\infty \sum_{\mu\in P} b_{\mu,k}\mathrm{ch}\,\bP_\mu$ with $b_{\mu,k}\in\mathbb{Z}$ satisyfing the finiteness properties described above (with $\mu\in P$ replacing $\beta\in Q$).
Of course, to obtain the the coefficients $a_{\mu,k}^\lambda$ in the expansion \eqref{StIntoProj}, one multiplies by $\ch\,\overline{\Delta}_\la$.

\subsubsection*{Projective into standard}

What is more special about the category $\fC^\star$ is that the projectives $\bP_\lambda$ admit $\Delta$-filtrations. At the level of characters, we have for each $\lambda\in P$ the convergent expansion
\begin{align*}
\mathrm{ch}\,\bP_\lambda = \sum_{\mu\in P} m_{\lambda\mu}(q)\mathrm{ch}\,\Delta_\mu
\end{align*}
of Corollary~\ref{CorPIntoE},
where $m_{\lambda\mu}(q)=\langle \ch\,\bP_\lambda, E_\mu(x^{-1},q^{-1},\infty)\rangle_0$.  Moreover, because this expansion arises from a filtration, it is necessarily {\em positive}, i.e., $m_{\lambda\mu}(q)\in\mathbb{Z}_{\ge 0}[q]$. (Since $m_{\lambda\mu}(q)$ is nothing but the coefficient of $x^{-\lambda}$ in $E_\mu(x^{-1},q^{-1},\infty)$, this positivity is also known by the combinatorial formula of \cite{OS}.)

Concretely, for each $\lambda\in P$ we obtain the following ``reciprocal'' version of the Macdonald identity:
\begin{align}
\label{ProjIntoSt2}
\frac{x^\lambda}{(q;q)_\infty^{\mathrm{rk}(\mathfrak{g})}\prod_{\alpha\in\Delta_+}(1-x^\alpha)\prod_{\alpha\in\Delta}(qx^\alpha;q)_\infty} = \sum_{\mu\in P} m_{\lambda\mu}(q)\frac{E_\mu(x,q,0)}{(q)_\mu}.
\end{align}
expressing the reciprocal of the product in \eqref{MacId} (times any monomial $x^\lambda$) as a non-negative linear combination of specialized nonsymmetric Macdonald polynomials. By considering the left-hand side of \eqref{ProjIntoSt2} directly, one realizes that the non-negativity (and convergence) of the right-hand side is quite surprising.

\begin{example}
Suppose $G=SL(2)$ and $\lambda=0$. Then \eqref{ProjIntoSt2} gives
\begin{align*}
&\mathrm{ch}\, \bP_0\\ 
&= E_0(x,q,0)+\frac{E_{\alpha}(x,q,0)}{1-q}+(q+q^2)\frac{E_{-\alpha}(x,q,0)}{(1-q)(1-q^2)}+(1+q+q^2)\frac{E_{2\alpha}(x,q,0)}{(1-q)(1-q^2)(1-q^3)}+\dotsm
\end{align*}
where $\alpha=\alpha_1$.\footnote{In type $A_1$, one may use the explicit formulas for $E_\mu(x,q,t)$ from \cite[\S1.2]{CO2}, for instance.} On the other hand, one may expand the left-hand side of \eqref{ProjIntoSt2} directly:
\begin{align*}
&\frac{1}{(q;q)_\infty(x^\alpha;q)_\infty(qx^{-\alpha};q)_\infty}\\ 
&\qquad = \frac{1}{\big((1-q)(1-q^2)\dotsm\big)\big((1-x^\alpha)(1-qx^{\alpha})\dotsm\big)\big((1-qx^{-\alpha})(1-q^2x^{-\alpha})\dotsm\big)}\\
&\qquad =(1+x^\alpha+x^{2\alpha}+x^{3\alpha}+\dotsm)+q(x^{-\alpha}+2+3x^\alpha+3x^{2\alpha}+\dotsm)+q^2\dotsm.
\end{align*}
Writing monomials in terms of specialized nonsymmetric Macondald polynomials
\begin{align*}
x^0 &= E_0(x,q,0)\\
x^\alpha &= E_\alpha(x,q,0)-q E_0(x,q,0)\\
x^{-\alpha} &= E_{-\alpha}(x,q,0)-E_\alpha(x,q,0)-E_0(x,q,0) \\
x^{2\alpha} &= E_{2\alpha}(x,q,0)-q^3E_{-\alpha}(x,q,0)-(q+q^2)E_{\alpha}(x,q,0)+q^3E_0(x,q,0)\\
x^{-2\alpha} &= E_{-2\alpha}(x,q,0)-E_{2\alpha}(x,q,0)-(1+q+q^2)E_{-\alpha}(x,q,0)\\&\qquad\qquad+(q+q^2)E_{\alpha}(x,q,0)+qE_{0}(x,q,0)\\
x^{3\alpha} &= E_{3\alpha}(x,q,0)-q^5E_{-2\alpha}(x,q,0)-(q+q^2+q^3+q^4)E_{2\alpha}(x,q,0)\\&\qquad\qquad+(q^5+q^6+q^7)E_{-\alpha}(x,q,0)+(q^3+q^4+q^5)E_\alpha(x,q,0)-q^6 E_0(x,q,0)
\end{align*}
one then arrives after much cancellation at
\begin{align*}
\frac{1}{(q;q)_\infty(x^\alpha;q)_\infty(qx^{-\alpha};q)_\infty} &=E_0+E_\alpha+E_{2\alpha}+\dotsm+q(E_{-\alpha}+E_{\alpha}+2E_{2\alpha}+\dotsm)+q^2\dotsm,
\end{align*}
which agrees exactly with \eqref{ProjIntoSt2} on the terms for $E_\mu(x,q,0) q^k$ with $k<2$ and $\mu\preceq 2\alpha$.
\end{example}

\section*{Declaration}
\subsection*{Ethical Approval} Not applicable
\subsection*{Competing interests} None 
\subsection*{Funding} D.O. gratefully acknowledges support from the Simons Foundation (Collaboration Grants for Mathematicians) and the Max Planck Institute for Mathematics (MPIM Bonn).
\subsection*{Data availability} None

\end{document}